\documentclass[12pt,dvips]{amsart}
\usepackage{amsfonts, amssymb, latexsym, epsfig, epsf, amsmath, amscd}
\usepackage{url}

\newtheorem{Theorem}{Theorem}[section]
\newtheorem{Proposition}[Theorem]{Proposition}
\newtheorem{Lemma}[Theorem]{Lemma}
\newtheorem{Claim}[Theorem]{Claim}

\newtheorem{Definition-Proposition}[Theorem]{Definition-Theorem}
\newtheorem{Main Conjecture}[Theorem]{Main Conjecture}

\newtheorem{Conjecture}[Theorem]{Conjecture}

\theoremstyle{remark}
\newtheorem{Example}[Theorem]{Example}
\newtheorem{Remark}[Theorem]{Remark}

\newcommand{\caO}{\mathcal{O}}

\newcommand{\field}{\mathbb}

\newcommand{\ga}{\alpha}

\newcommand{\C}{{\field C}}

\newcommand{\Z}{{\field Z}}

\theoremstyle{plain}

\newtheorem{Question}{Question}

\newcommand{\excise}[1]{}

\newcommand{\comment}[1]{$\star${\sf\textbf{#1}}$\star$}

\newcommand{\cellsize}{11}
\newlength{\cellsz} 
\newsavebox{\cell}
\sbox{\cell}{\begin{picture}(\cellsize,\cellsize)
\put(0,0){\line(1,0){\cellsize}}
\put(0,0){\line(0,1){\cellsize}}
\put(\cellsize,0){\line(0,1){\cellsize}}
\put(0,\cellsize){\line(1,0){\cellsize}}
\end{picture}}
\newcommand\cellify[1]{\def\thearg{#1}\def\nothing{}%
\ifx\thearg\nothing
\vrule width0pt height\cellsz depth0pt\else
\hbox to 0pt{\usebox{\cell} \hss}\fi%
\vbox to \cellsz{
\vss
\hbox to \cellsz{\hss$#1$\hss}
\vss}}
\newcommand\tableau[1]{\vtop{\let\\\cr
\baselineskip -16000pt \lineskiplimit 16000pt \lineskip 0pt
\ialign{&\cellify{##}\cr#1\crcr}}}
\newcommand{\kellsize}{24}
\newlength{\kellsz} 
\newsavebox{\kell}
\sbox{\kell}{\begin{picture}(\kellsize,\kellsize)
\put(0,0){\line(1,0){\kellsize}}
\put(0,0){\line(0,1){\kellsize}}
\put(\kellsize,0){\line(0,1){\kellsize}}
\put(0,\kellsize){\line(1,0){\kellsize}}
\end{picture}}
\newcommand\kellify[1]{\def\thearg{#1}\def\nothing{}%
\ifx\thearg\nothing
\vrule width0pt height\kellsz depth0pt\else
\hbox to 0pt{\usebox{\kell} \hss}\fi%
\vbox to \kellsz{
\vss
\hbox to \kellsz{\hss$#1$\hss}
\vss}}
\newcommand\ktableau[1]{\vtop{\let\\\cr
\baselineskip -16000pt \lineskiplimit 16000pt \lineskip 0pt
\ialign{&\kellify{##}\cr#1\crcr}}}

\newcommand{\sellsize}{63}
\newlength{\sellsz} 
\newsavebox{\sell}
\sbox{\sell}{\begin{picture}(\sellsize,20)
\put(0,0){\line(1,0){\sellsize}}
\put(0,0){\line(0,1){\sellsize}}
\put(\sellsize,0){\line(0,1){\sellsize}}
\put(0,\sellsize){\line(1,0){\sellsize}}
\end{picture}}
\newcommand\sellify[1]{\def\thearg{#1}\def\nothing{}%
\ifx\thearg\nothing
\vrule width0pt height\sellsz depth0pt\else
\hbox to 0pt{\usebox{\sell} \hss}\fi%
\vbox to \sellsz{
\vss
\hbox to \sellsz{\hss$#1$\hss}
\vss}}
\newcommand\stableau[1]{\vtop{\let\\\cr
\baselineskip -16000pt \lineskiplimit 16000pt \lineskip 0pt
\ialign{&\sellify{##}\cr#1\crcr}}}

\begin{document}
\pagestyle{plain}
\mbox{}
\title{Polynomials for $GL_p\times GL_q$ orbit closures in the flag variety}
\author{Benjamin J.~Wyser}
\author{Alexander Yong}
\address{Department of Mathematics, University of Illinois at
Urbana-Champaign, Urbana, IL 61801, USA}
\email{bwyser@uiuc.edu, ayong@uiuc.edu}
\date{March 4, 2014}
\maketitle

\begin{abstract}
The subgroup $K=GL_p \times GL_q$ of $GL_{p+q}$ 
acts on the (complex) flag variety $GL_{p+q}/B$ with finitely many orbits. We introduce a family of polynomials specializing
to representatives for cohomology classes of the orbit
closures in the Borel model.  We define and study 
$K$-orbit determinantal ideals to support the geometric naturality of these representatives.
Using a modification of these ideals, we describe an analogy between two local singularity measures:
the $H$-polynomials and the Kazhdan-Lusztig-Vogan polynomials.
\end{abstract}

\tableofcontents

\newpage 

\section{Introduction}

\subsection{Polynomial representatives in ordinary cohomology}
Consider the Levi subgroup $K=GL_p \times GL_q$ of $GL_n$ ($n=p+q$).  (Throughout, we consider only
complex general linear groups.)
By a general result of T.~Matsuki \cite[Theorem 3]{Matsuki}, the flag variety $GL_n/B$ decomposes as a disjoint
union of finitely many $K$-orbits:
\[GL_n/B=\coprod_{\gamma} {\mathcal O}_{\gamma}.\]
The orbits ${\mathcal O}_{\gamma}$ are parameterized by
{\bf $(p,q)$-clans} $\gamma$, as described first by T.~Matsuki-T.~Oshima \cite[Theorem 4.1]{Matsuki-Oshima}, and later elaborated
upon by A.~Yamamoto \cite[Theorem 2.2.8]{Yamamoto}.  These clans are partial matchings of vertices $\{1,2,\ldots,n\}$, where
unpaired vertices are assigned $+$ or $-$; the difference in the number of $+$'s and $-$'s must be $p-q$.
Let ${\tt Clans}_{p,q}$ denote the set of all such clans.  Three clans from ${\tt Clans}_{6,4}$ are shown
below:

\[
\begin{picture}(280,10)
\put(0,0){\epsfig{file=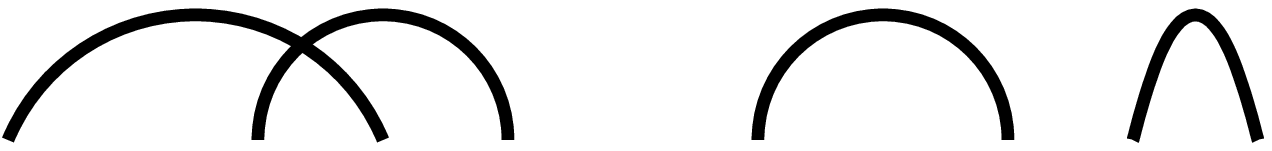, height=.8cm}}
\put(8,0){$+$}
\put(17,0){$-$}
\put(26,0){$+$}
\put(43,0){$+$}
\put(64,0){$-$}
\put(83,0){$+$,}

\put(104,0){$+$}
\put(120,0){$-$}
\put(130,0){$+$}
\put(139,0){$-$}
\put(159,0){$+$}
\put(177,0){$+$}
\put(196,0){,}
\put(210,0){$-\!+\!+\!+\!-\!+\!-\!-\!+\!+$}
\end{picture}
\]

Let $Y_{\gamma}$ be the Zariski closure of ${\mathcal O}_{\gamma}$. This is the union of ${\mathcal O}_{\beta}$ for $\beta\prec \gamma$,
where (by definition) $\prec$ is the {\bf closure order} on clans.
It is an irreducible variety.  By the formula of \cite[Proposition 2.3.8]{Yamamoto}, its dimension is 
${p\choose 2}+{q\choose 2}+\ell(\gamma)$ where
\begin{equation}
\label{eqn:Yam}
\ell(\gamma)=\sum_{\mbox{\tiny{vertices $i<j$ are matched}}}
j-i-\#\{\mbox{matchings of $s<t$ where $s<i<t<j$}\}.
\end{equation}

$Y_{\gamma}$ admits a class in singular cohomology (with $\Z$ coefficients):
\[[Y_{\gamma}]\in H^{\star}(GL_n/B)\cong {\mathbb Z}[x_1,\ldots,x_n]/I^{S_n},\]
where $I^{S_n}$ is the ideal generated by symmetric polynomials without constant term.  The above isomorphism, due to A.~Borel \cite{Borel} (cf. \cite[Section~10.2]{Fulton}),
is suggestive of the following problem:

\begin{quote}
Describe a choice of polynomial representatives $\{\Upsilon_{\gamma}\}$ 
for the cosets associated to
$\{[Y_{\gamma}]\}$ under Borel's isomorphism.
\end{quote}

One solution begins by assigning polynomials to the 
${n\choose p}$-many {\bf closed orbits}.
These orbits are indexed by {\bf matchless clans} $\tau$, i.e., those consisting of $p$ many $+$'s and $q$ many $-$'s (the third displayed
clan above is an example).  We will typically use $\tau$ to denote a matchless clan, and $\gamma$ to indicate an
arbitrary clan.
The {\bf divided difference operator} $\partial_i:{\mathbb Z}[x_1,\ldots,x_n]\to {\mathbb Z}[x_1,\ldots,x_n]$ is
\[\partial_i f=\frac{f-f^{s_i}}{x_i-x_{i+1}}.\]
Representatives for all
other orbits can be obtained by recursion using the $\partial_i$'s along a choice of path 
in \emph{weak order} (defined in Section \ref{sec:defs}).  
This approach was used by the first author in \cite{Wyser-13a}.

We consider a different choice of polynomial representatives for the closed orbits than that found in \emph{loc.~cit}.
From our perspective, this alternative choice of representatives is preferable for the following reasons:
\begin{itemize}
\item It is provably ``self-consistent'', by which we mean that each $\Upsilon_{\gamma}$ is a well-defined polynomial.
Specifically, $\Upsilon_{\gamma}$ depends neither on the choice of closed $K$-orbit $\caO_{\tau}$
at which we start the recurrence, nor on the aforementioned 
choice of path in weak order.
 \item Each $\Upsilon_{\gamma}$ has nonnegative integer coefficients, and in many cases the geometric reason for this
is transparent.
 \item Our choice extends simply to $T$-equivariant cohomology and ($T$-equivariant) 
$K$-theory, where $T$ is the torus of diagonal matrices in $GL_n$.  (\cite{Wyser-13a} covers the case of
$T$-equivariant cohomology, but neither ordinary nor $T$-equivariant $K$-theory are discussed.)  We mostly 
suppress discussion of these refinements until Section~2.
\end{itemize}

To formulate our answer, we associate to a matchless $(p,q)$-clan $\tau$ a partition, which we will denote $\lambda(\tau)$. We will also associate a sequence of nonnegative integers denoted by ${\vec f}(\tau)$; this sequence is called
a ``flagging'' in the context that we will use it below.

The partition $\lambda(\tau)$ is formed as follows.  Start from the upper-right corner of a $p \times q$
rectangle, and trace a lattice path to the lower-left corner, by moving down at step $i$ if the $i$th character
of $\tau$ is a $+$, and left if it is a $-$.  Then $\lambda(\tau)$ is the partition whose Young diagram is the
portion of the $p \times q$ rectangle northwest of this path.  Clearly,
the assignment
of $\lambda(\tau)$ to $\tau$ defines a bijection between matchless $(p,q)$-clans (or, equivalently, $p$-element
subsets of $\{1,\hdots,n\}$) and partitions whose Young diagrams fit within a $p \times q$ rectangle.

Now, ${\vec f}(\tau) = (f_1,\hdots,f_p)$ for $\lambda(\tau)$ is defined by $f_i = \text{index of $i$th
$+$ of $\tau$}$.

Next, let $\widehat\tau$ denote the $(q,p)$-clan obtained from $\tau$ by flipping all signs.  Then we can also 
form the partition $\lambda(\widehat\tau)$ and the flagging ${\vec f}(\widehat\tau)$, as described above.
Note that this partition has $q$ parts, and its flagging is a $q$-tuple.

As an example, if $\tau=++--+-++$ then $\lambda(\tau)=(3,3,1,0,0)$ and ${\vec f}(\tau)=(1,2,5,7,8)$, while
$\lambda(\widehat\tau)=(3,3,2)$ and ${\vec f}(\widehat\tau)=(3,4,6)$.  The relevant pictures are as follows:

\begin{figure}[h]
\begin{picture}(220,75)
\thinlines
\put(0,0){\makebox[0pt][l]{\framebox(45,75)}}
\put(15,0){\line(0,1){75}}
\put(30,0){\line(0,1){75}}
\put(0,15){\line(1,0){45}}
\put(0,30){\line(1,0){45}}
\put(0,45){\line(1,0){45}}
\put(0,60){\line(1,0){45}}

\linethickness{2pt}
\put(45,76){\line(0,-1){30}}
\put(46,45){\line(-1,0){30}}
\put(16,46){\line(0,-1){16}}
\put(17,30){\line(-1,0){17}}
\put(0,31){\line(0,-1){32}}

\put(47,66){$+$}
\put(47,51){$+$}
\put(33,38){$-$}
\put(18,38){$-$}
\put(17,31){$+$}
\put(3,23){$-$}
\put(2,16){$+$}
\put(2,4){$+$}

\put(62,65){$1$}
\put(62,50){$2$}
\put(62,35){$5$}
\put(62,20){$7$}
\put(62,5){$8$}


\thinlines
\put(120,15){\makebox[0pt][l]{\framebox(75,45)}}
\put(135,15){\line(0,1){45}}
\put(150,15){\line(0,1){45}}
\put(165,15){\line(0,1){45}}
\put(180,15){\line(0,1){45}}
\put(120,30){\line(1,0){75}}
\put(120,45){\line(1,0){75}}

\linethickness{2pt}
\put(196,60){\line(-1,0){30}}
\put(166,61){\line(0,-1){31}}
\put(166,31){\line(-1,0){16}}
\put(151,31){\line(0,-1){16}}
\put(151,16){\line(-1,0){31}}
\put(185,52){$-$}
\put(170,52){$-$}
\put(167,46){$+$}
\put(167,33){$+$}
\put(154,22){$-$}
\put(152,16){$+$}
\put(138,7){$-$}
\put(123,7){$-$}
\put(200,50){$3$}
\put(200,35){$4$}
\put(200,20){$6$}
\end{picture}
\caption{$\lambda(\tau),{\vec f}(\tau)$ and 
$\lambda(\widehat \tau),{\vec f}(\widehat \tau)$ for $\tau=++--+-++$.} 
\end{figure}
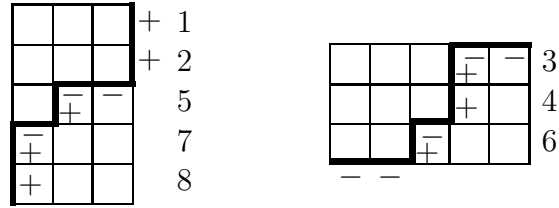

Now, given any partition $\lambda=(\lambda_1\geq\lambda_2\geq\cdots\geq \lambda_m\geq 0)$ and a sequence of nonnegative integers
$\vec f=(f_1,\ldots,f_m)$ (a flagging), one defines the {\bf flagged Schur polynomial} to be
\[s_{\lambda,\vec f}(X)=\sum_T {\bf x}^T,\]
where the sum is over all semistandard tableaux $T$ of shape $\lambda$ whose entries in row
$i$ are weakly bounded above by $f_i$; see \cite[Section~2.6]{Manivel} for a textbook treatment of flagged Schur polynomials. So considering the partition $\lambda(\tau)=(3,3,1,0,0)$ and the flagging
${\vec f}(\tau) = (1,2,5,7,8)$ coming from the clan $\tau$ in our example, 
\[s_{(3,3,1,0,0),(1,2,5,7,8)}(x_1,x_2,x_3,x_4,x_5)=x_1^3 x_2^3 x_3+x_1^3 x_2^3 x_4 + x_1^3 x_2^3 x_5.\]
The three monomials correspond to the tableaux
\[\tableau{1 & 1& 1\\ 2& 2&2 \\3}, \ \ \  
\tableau{1 & 1& 1\\ 2& 2&2 \\4} \ \ \ \mbox{and  \ \ \ }
\tableau{1 & 1& 1\\ 2& 2&2 \\5}.
\]
Since $\lambda(\tau)$ and ${\vec f}(\tau)$ are determined by $\tau$,
one can use the abbreviation ${\mathfrak s}_{\tau}(X) = s_{\lambda(\tau),{\vec f}(\tau)}(X)$, and similarly define
${\mathfrak s}_{\widehat\tau}(X)$.  For matchless $\tau$, define:
\begin{equation}
\label{eqn:defformatchless}
\Upsilon_{\tau}:={\mathfrak s}_{\tau}(X) \cdot {\mathfrak s}_{\widehat\tau}(X).
\end{equation}
Given a clan $\gamma$ which is not matchless, by \cite[Theorem 4.6]{Richardson-Springer} there exists a matchless
clan $\tau$ and a sequence $s_1,\hdots,s_l$
of simple transpositions such that $\gamma = s_1 \cdot s_2 \cdot \hdots s_l \cdot \tau$.  (This notation 
is explained in Section \ref{sec:defs}.)  In general, neither $\tau$ nor the permutation
$w=s_1 \hdots s_l$ is uniquely determined by $\gamma$.  Our wish is to define
\[ \Upsilon_{\gamma}=\partial_1 \hdots \partial_l \Upsilon_{\tau}. \]
However, in light of the preceding sentence, it is not at all clear that this is a valid ``definition''.  The main
purpose of this paper is to present the following (and its refinements):
\begin{Theorem}
\label{thm:intro}
Each $\Upsilon_{\gamma}$ is well-defined and represents $[Y_{\gamma}]$ under Borel's isomorphism. 
\end{Theorem}

We now make a few easy observations about the
$\Upsilon_{\gamma}$'s.

The flagged Schur polynomials from (\ref{eqn:defformatchless}) are 
Schubert polynomials (see Proposition~\ref{prop:KMY}).
It is a standard fact that any product of Schubert classes expands as a nonnegative linear combination of Schubert classes, and moreover the Schubert polynomials represent the
Schubert classes under the Borel isomorphism; see, e.g., Chapter~10 (and specifically Section~10.4) of
\cite{Fulton}. It follows that
$\Upsilon_{\gamma}$ is a nonnegative linear combination of Schubert polynomials. 
Since Schubert polynomials have nonnegative integer coefficients,
\[\Upsilon_{\gamma}\in {\mathbb Z}_{\geq 0}[x_1,\ldots,x_n]
\mbox{ \ \  for all $\gamma\in {\tt Clan}_{p,q}$.}\]
We have emphasized the monomial expansion of $\Upsilon_{\gamma}$ since this positivity should have a geometric explanation (see Section~3).

Finally, by our definition of $\Upsilon_{\tau}$ for $\tau$ matchless, 
it is easy to see that $\Upsilon_{\tau}$ has degree
$\binom{n}{2} - \binom{p}{2} - \binom{q}{2} = pq$.  This reflects the 
fact that any closed $K$-orbit is isomorphic
to the flag variety for the group $K$, and hence has dimension equal to $\binom{p}{2} + \binom{q}{2}$.  Combining
this with the aforementioned dimension formula of A.~Yamamoto (cf.~(\ref{eqn:Yam})), 
and with the fact that application of $\partial_i$ lowers the degree of any 
polynomial by $1$, it follows that the degree of $\Upsilon_{\gamma}$ 
for arbitrary
$\gamma$ is $pq - l(\gamma)$,
the codimension of $\caO_{\gamma}$ in the flag variety.

\subsection{Further results and comparisons to the literature}
For a reductive algebraic group $G$ over ${\mathbb C}$,
let $B$ be a Borel subgroup and $K\subset G$ be a {\bf spherical} subgroup,
i.e., one which acts by left translations on $G/B$ with finitely many orbits.

The most widely analyzed case is when $K=B$, where the orbit closures
are Schubert varieties. In this setting, the polynomial representatives 
problem was studied for Schubert varieties (in general type) by 
I.~Bernstein-I.~Gelfand-S.~Gelfand \cite{BGG}. 
In type $A$, this led to the development of
\emph{Schubert polynomials} by A.~Lascoux-M.-P.~Sch\"{u}tzenberger
\cite{Lascoux.Schutzenberger}. Both papers begin with a 
choice of polynomial representative for the class of a point, with the remainder recursively obtained using $\partial_i$'s. However, the salient feature of Schubert polynomials is the
nonnegativity of their coefficients.  Since their discovery, many nice combinatorial properties of Schubert polynomials
have been found, including combinatorial formulas for their expansion; see, 
e.g., the textbook \cite{Manivel}. 
We will use properties of Schubert polynomials to establish our main results.

A spherical subgroup $K$ is {\bf symmetric} if $K=G^{\theta}$ is the fixed point 
subgroup for a holomorphic
involution $\theta$ of $G$. The symmetric pairs $(G,K)$ have a classification.  For generalities, the reader may consult, e.g., \cite{Matsuki,Springer,Matsuki-Oshima,Richardson-Springer}.  The case of 
$(GL_{p+q},GL_p \times GL_q)$ corresponds to the involution 
\[\theta(A)=I_{p,q}A I_{p,q}\]
where
$I_{p,q}$ is the diagonal $\pm 1$ matrix with $p$ many $1$'s followed by $q$ many $-1$'s. For more details about
this case, see, e.g., \cite{Yamamoto,McGovern,McGovern-Trapa,Wyser-13a}.  

The first author gave equivariant cohomology representatives 
for the closed orbits of cases of symmetric pairs $(G,K)$ with $G$ classical in \cite{Wyser-13a, Wyser-13b}.  For the case of $(GL_{p+q}, GL_p\times GL_q)$, small examples suggest that those representatives may also produce a self-consistent system, although we do not know a proof of this.  At any rate, their ordinary cohomology specializations do not have nonnegative integer coefficients in general. 

To our best knowledge,
this paper provides the first self-consistency proof of its kind for any symmetric pair $(G,K)$. In the case of
Schubert varieties, the divided difference recurrence has only one initial condition (the class of a point).
Further, minimal paths in the weak Bruhat order of $S_n$ correspond to reduced words of the same permutation.
Since divided differences satisfy the braid relations, self-consistency is automatic for Schubert polynomials.
As we have observed, neither of these two helpful properties hold for the symmetric pair we consider
here.  

For some other symmetric pairs (also defined over the complex numbers), such as $(GL_{2n},Sp_{2n})$ or $(GL_n,O_n)$, the property of having only one initial condition --- that is, a unique closed $K$-orbit --- \textit{does} hold.  However, even in such cases,
minimal chains in weak order can again correspond to reduced words of \textit{different} permutations, so
self-consistency is not a given in these cases either.  The two aforementioned additional cases are considered
in a sequel \cite{WyYo2}.

There is further support for the choice of $\Upsilon_{\gamma}$. We 
use a geometric perspective originally applied by
A.~Knutson-E.~Miller \cite{Knutson.Miller:annals} to justify Schubert polynomials.
For a variety $X\subset GL_n/B$, consider the preimage $\pi^{-1}(X)\subset GL_n$ under the natural projection,
and $\overline{\pi^{-1}(X)}\subset {\rm Mat}_{n\times n}$. Because $\pi^{-1}(X)$ is a union of
left cosets of $B$, $\overline{\pi^{-1}(X)}$ is stable under right multiplication by $B$.
Identifying
\[[\overline{\pi^{-1}(X)}]_B\in H^\star_B({\rm Mat}_{n\times n}) \mbox{\  with 
$[\overline{\pi^{-1}(X)}]_T \in H^\star_T({\rm Mat}_{n\times n})\cong {\mathbb Z}[x_1,\ldots, x_n]$}\]
(see \cite[Section~1.2]{Knutson.Miller:annals}) uniquely picks out a polynomial representative for $[X]\in H^{\star}(GL_n/B)$. In the case $X=X_w:={\overline{B_{-}wB/B}}$ of Schubert varieties,
to actually compute $[\overline{\pi^{-1}(X_w)}]_T$ they obtain, by Gr\"{o}bner degeneration, the multidegree of
Fulton's \emph{Schubert determinantal ideal} $I_w$, whose generators scheme-theoretically cut out $\overline{\pi^{-1}(X_w)}$.
Their conclusion is 
\[[\overline{\pi^{-1}(X_w)}]_T={\mathfrak S}_w(x_1,\ldots,x_n),\] 
the Schubert polynomial for $X_w$ \cite[Theorem A]{Knutson.Miller:annals}.

To study the case $X=Y_{\gamma}$, we define the \emph{$K$-orbit determinantal ideal} $I_{\gamma}$, generated by minors of the generic $n\times n$ matrix and certain auxiliary matrices.   
When $\gamma$ is {\bf non-crossing}, i.e., no two arcs overlap (see the second of the displayed clans on page~2 for a non-example),
these generators form a Gr\"{o}bner basis with squarefree lead terms. The prime decomposition of the Gr\"{o}bner limit is 
indexed by monomials of $\Upsilon_{\gamma}$. That is, $I_{\gamma}$
scheme-theoretically cuts out $\overline{\pi^{-1}(Y_\gamma)}$, and we show
\[[\overline{\pi^{-1}(Y_\gamma)}]_T=\Upsilon_{\gamma}(x_1,\ldots,x_n), \text{\ \  for non-crossing $\gamma$}.\]
See Theorem \ref{claim:main}, whose proof uses \cite{Knutson.Miller:annals, KMY,
Wyser12, Wyser-13a}.  This provides a geometric rationale for 
our choice of representatives, at least for the non-crossing case.  Furthermore, we conjecture that the above
equality holds for all $\Upsilon_{\gamma}$, whether $\gamma$ has crossings or not (cf. Section \ref{sec:conjectures}).  

	The non-crossing condition is special because then
$Y_{\gamma}$ is a \emph{Richardson variety} \cite{Wyser12}, or the intersection of a Schubert variety with an
opposite Schubert variety.  Such varieties are so named because they were first studied by R.~W.~Richardson in
\cite{Richardson}.  Properties of Richardson varieties can be transparently deduced from the two Schubert
varieties involved \cite{KWY}. These facts were our starting point for this project.

In \cite{Brion}, M.~Brion proves (in a general setting, which applies in particular to the case at hand)
a formula for $[Y_{\gamma}]$ as a sum of Schubert \emph{classes}.  In our example, this sum turns out to be
multiplicity-free, meaning that all Schubert classes occurring in the sum occur with coefficient $1$.
Thus taking Brion's formula and replacing each Schubert \textit{class} with its corresponding Schubert \textit{polynomial}
gives a cohomological representative of the type we are seeking.  Indeed, our arguments will make it apparent that
the representative so obtained is in fact equal to $\Upsilon_{\gamma}$. However, while Brion's formula applies in both
(ordinary) cohomology and $K$-theory, it does \textit{not} apply $T$-equivariantly in either theory.  Thus our 
representatives in the $T$-equivariant setting are truly ``new``, in the sense that they cannot be easily be
deduced from Brion's formula.

Finally, we consider a modification of the $K$-orbit determinantal ideal which we conjecture  
provides local equations of $Y_{\gamma}$, cf. Conjecture \ref{conj:D}.  Having such equations allows us to study
the singularities of the orbit closures inside $G/B$.  The \emph{Kazhdan-Lusztig-Vogan polynomials} are one local
measure of these singularities. 
We describe a conjectural analogy with another singularity measure, the \emph{$H$-polynomials} of $Y_{\gamma}$, 
defined in Section \ref{sec:h-poly}.
This analogy parallels that between \emph{Kazhdan-Lusztig polynomials} and $H$-polynomials of Schubert
varieties described by L.~Li and the second author in \cite[Section~2]{Li.Yong}.

\subsection{Organization}
In Section~2, we introduce a family of polynomials in two sets of variables, with a deformation parameter. This family is 
defined using Schubert polynomials and divided difference operators. With this,
we state our choice of polynomial representatives
for equivariant cohomology and equivariant $K$-theory. We establish our main theorems (Theorems~\ref{thm:intro},~\ref{thm:equivver}
and~\ref{thm:K}) that
they define a self-consistent system. In Section~3, we define the
$K$-orbit determinantal ideal and establish our Gr\"{o}bner basis theorem in the non-crossing case as well as formulate the more general conjectures. 
In Section~4, we use a modification of these ideals in our exploration of the singularities of $Y_{\gamma}$.

\section{More polynomial families and cohomology theories}

\subsection{Definition of $\Upsilon^{(\beta)}_{\gamma}$}\label{sec:defs}
For non-crossing $\gamma$, define $u(\gamma)\in S_n$ by assigning 
\begin{itemize}
\item $-$'s and left endpoints of arcs the labels $1,2,\ldots,q-1,q$ from left to right, and
\item $+$'s and right endpoints of arcs the labels $q+1,q+2,\ldots,n$ from left to right.
\end{itemize}  
Define $v(\gamma)\in S_n$ by assigning 
\begin{itemize}
\item $+$'s and left endpoints of arcs the labels $1,2,\ldots,p-1,p$ from left to right, and 
\item $-$'s and right endpoints of arcs the labels $p+1,p+2,\ldots,n$ from left to right. 
\end{itemize}

\begin{Example}
For the second clan $\gamma\in {\tt Clan}_{6,4}$ shown on page~2,
$u(\gamma)=512637849 \ 10$ and $v(\gamma)=127389456 \ 10$.
\end{Example}

\begin{Example}
\label{exa:nonmatching}
We are especially interested in matchless clans, which we typically denote by $\tau$. If $\tau=++--+-++$
(as in Section~1) then $u(\tau)=45126378\in S_8$ (in one-line notation) and $v(\tau)=12673845$.\qed
\end{Example}

The discussion that follows freely uses facts about Schubert varieties, flag varieties and Schubert polynomials. Material on Schubert varieties and flag varieties may be found in Chapters 9 and 10 of \cite{Fulton}. 
Material about Schubert polynomials appears in Chapter 10.4 of \emph{loc. cit} as well as Chapter 2 of \cite{Manivel}.

Let $X=\{x_1,x_2,\ldots,x_n\}$ and $Y=\{y_1,y_2,\ldots,y_n\}$ be independent and commuting
indeterminates. The $\beta$-{\bf double Schubert polynomial} ${\mathfrak S}^{(\beta)}_w(X;Y)$ is defined by setting
\[{\mathfrak S}^{(\beta)}_{w_0}(X;Y)=\prod_{i=1}^{n-1}\prod_{j=1}^{n-i}(x_i-y_j+\beta x_i y_j)\]
where $w_0$ is the long element of $S_n$.  Define $\partial^{(\beta)}_i$ by
\[\partial^{(\beta)}_i(f)=\partial_i((1-\beta x_{i+1})f).\]
Now, if $i$ is any position such that $w(i)<w(i+1)$ then 
\[{\mathfrak S}^{(\beta)}_w(X;Y)=\partial^{(\beta)}_i{\mathfrak S}^{(\beta)}_{ws_i}(X;Y)\] 
where $s_i$ is the simple reflection transposing $i$ and $i+1$. 
Recall that 
\[{\mathfrak S}_{w}(X;Y)={\mathfrak S}_{w}^{(0)}(X,Y)\] 
is the {\bf double Schubert polynomial} and
\[{\mathfrak S}_w(X)={\mathfrak S}_w^{(0)}(X;0)\]
is the  {\bf single Schubert polynomial}. Also,
\[{\mathfrak G}_w(X;Y)={\mathfrak S}^{(1)}_w\left(x_i\mapsto 1-x_i;y_j\mapsto \frac{1-y_j}{y_j}\right)\]
is the {\bf double Grothendieck 
polynomial} ${\mathfrak G}_w(X;Y)$ and finally,
\[{\mathfrak G}_w(X)={\mathfrak G}_w(X;y_j\mapsto 1)\]
is the {\bf single Grothendieck polynomial}. The use of a deformation parameter $\beta$ in Schubert polynomial theory is found in \cite{Fomin.Kirillov:groth}.
Below we remind the reader in what sense  the above substitutions give
representatives of the Schubert classes.

When $\tau$ is matchless, define
\begin{equation}
\label{eqn:matchlessdef}
\Upsilon^{(\beta)}_\tau(X,Y)={\mathfrak S}^{(\beta)}_{u(\tau)}(X;y_n,y_{n-1},\ldots,y_1)\cdot
								{\mathfrak S}^{(\beta)}_{v(\tau)}(X;Y).
\end{equation}

For clans $\gamma$ which are not matchless, $\Upsilon_{\gamma}$ will be defined using divided difference operators
according to the weak order on $K$-orbits, which we now define.  Geometrically, we say that an orbit closure $Y_{\gamma}$
covers another orbit closure $Y_{\gamma'}$, and write $\gamma = s_i \cdot \gamma'$, if 
\[ Y_{\gamma} = \pi_i^{-1}(\pi_i(Y_{\gamma'})), \]
where $\pi_i: G/B \rightarrow G/P_{s_i}$ is the natural projection.  Here, $P_{s_i}$ is the standard minimal parabolic
subgroup $B \cup Bs_iB$ of $G$.  Note that this definition makes sense not only in our current example, but in any
situation where we are dealing with varieties $Y$ which are closures of orbits of a spherical subgroup acting on
$G/B$.  Indeed, this is the appropriate definition of weak order in all such settings.

In our example, the weak order has the following combinatorial description \cite{Matsuki, Yamamoto}.  The
{\bf weak Bruhat order} on ${\tt Clans}_{p,q}$ is the transitive
closure of the covering relation $s_i \cdot \gamma \succ \gamma=(c_1,\ldots,c_n)$ if either:
\begin{itemize}
\item[(a)] $s_i \cdot \gamma=(\ldots,c_{i+1},c_i,\ldots)$ and
\begin{itemize}
\item[$\bullet$] $c_i$ is a sign and  $c_{i+1}$ is the end of an arc matching
with a vertex to its right;
\item[$\bullet$] $c_i$ is the end of an arc matching with a vertex to its left
and $c_{i+1}$ is a sign; or
\item[$\bullet$] $c_i$ and $c_{i+1}$ are endpoints of different arcs, and the mate of $c_i$ is left of the
mate of $c_{i+1}$
\end{itemize}
\item[(b)] $s_i \cdot \gamma$ is obtained from $\gamma$ by replacing
$c_i=\pm$ and $c_{i+1}=\mp$ by an arc.
\end{itemize}

If $\gamma$ is not matchless, it follows from \cite[Theorem 4.6]{Richardson-Springer} that there is a
matchless clan $\tau$ and a sequence of the form
\[ \gamma = s_1 \cdot s_2 \cdots s_l \cdot \tau. \]

(Here, $l = l(\gamma)$ in the notation of Section~1.)  In this event, let
\[\Upsilon^{(\beta)}_\gamma(X;Y) = \partial_1^{(\beta)} \hdots \partial_l^{(\beta)} \Upsilon^{(\beta)}_{\tau}(X;Y).\]
Just as representatives of Schubert classes are specializations of ${\mathfrak S}^{(\beta)}_w(X;Y)$,
we will see that the same specializations of $\Upsilon^{(\beta)}(X;Y)$ give representatives of the classes
of $Y_{\gamma}$'s:
\begin{eqnarray}\nonumber
\Upsilon_\gamma(X;Y) & := & \Upsilon^{(0)}_{\gamma}(X;Y)\\ \nonumber
\Upsilon_{\gamma}(X) & := & \Upsilon_{\gamma}(X;0)\\ \nonumber
\Upsilon^K_\gamma(X;Y) & := & \Upsilon^{(1)}_{\gamma}\left(x_i\mapsto 1-x_i;y_j\mapsto \frac{1-y_j}{y_j}\right)\\ \nonumber
\Upsilon^K_{\gamma}(X) & := & \Upsilon^K_{\gamma}(X;y_j\mapsto 1)\nonumber
\end{eqnarray}

\subsection{Some combinatorial properties of $\Upsilon^{(\beta)}_{\gamma}$}
We assume familarity with standard permutation combinatorics such as the Rothe diagram, essential set, code of a permutation and pattern avoidance; see, e.g.,
\cite[Sections 2.1-2.2]{Manivel}.

A permutation is {\bf vexillary} if it is $2143$-avoiding.
\begin{Lemma}
\label{lemma:isvex}
If $\gamma$ is non-crossing, then $u(\gamma)$ and $v(\gamma)$ are vexillary permutations. In addition $u(\gamma)$ and $v(\gamma)$
are inverse to Grassmannian permutations with descents at $q$ and $p$ respectively.
\end{Lemma}
\begin{proof}
Consider $u:=u(\gamma)$ and suppose $i_1<i_2<i_3<i_4$ where
$u(i_1),u(i_2),u(i_3),u(i_4)$ are in the relative order $2143$. Then since $1,2,\ldots,q$ and $q+1,q+2,\ldots,p+q$
form rising sequences in $u$,
$\gamma(i_1),\gamma(i_3)\in \{q+1,q+2,\ldots,p+q\}$ and  $\gamma(i_2),\gamma(i_4)\in \{1,2,\ldots,q\}$.
Hence $\gamma(i_1)>\gamma(i_4)$, a contradiction. Thus $u$ is vexillary.

It is straightforward to see that the essential set of $u$ (provided $u$ is not the identity) must all lie in column $q$. This is equivalent
to the inverse Grassmannian claim.

The arguments for $v(\gamma)$ are similar. 
\end{proof}

\begin{Example}
Continuing Example~\ref{exa:nonmatching}, where $\tau=++--+-++$,
the diagrams of $u(\tau)$ and $v(\tau)$ are given below. (The $\bullet$'s of $D(\pi)$ are in
positions $(i,\pi(i))$.)
\[\begin{picture}(330,96)

\put(0,44){$D(u(\tau))=$}
\put(60,0){\makebox[0pt][l]{\framebox(96,96)}}
\thicklines
\put(60,72){\makebox[0pt][l]{\framebox(36,24)}}
\put(72,96){\line(0,-1){24}}
\put(84,96){\line(0,-1){24}}
\put(60,84){\line(1,0){36}}
\put(84,36){\makebox[0pt][l]{\framebox(12,12)}}
\thinlines
\put(102,90){\circle*{4}}
\put(102,90){\line(1,0){54}}
\put(102,90){\line(0,-1){90}} 

\put(114,78){\circle*{4}}
\put(114,78){\line(1,0){42}}
\put(114,78){\line(0,-1){78}} 

\put(66,66){\circle*{4}}
\put(66,66){\line(1,0){90}}
\put(66,66){\line(0,-1){66}} 

\put(78,54){\circle*{4}}
\put(78,54){\line(1,0){78}}
\put(78,54){\line(0,-1){54}} 

\put(126,42){\circle*{4}}
\put(126,42){\line(1,0){30}}
\put(126,42){\line(0,-1){42}} 

\put(90,30){\circle*{4}}
\put(90,30){\line(1,0){66}}
\put(90,30){\line(0,-1){30}}

\put(138,18){\circle*{4}}
\put(138,18){\line(1,0){18}}
\put(138,18){\line(0,-1){18}} 

\put(150,6){\circle*{4}}
\put(150,6){\line(1,0){6}}
\put(150,6){\line(0,-1){6}} 


\put(170,44){$D(v(\tau))=$}
\put(230,0){\makebox[0pt][l]{\framebox(96,96)}}
\thicklines
\put(254,48){\makebox[0pt][l]{\framebox(36,24)}}
\put(266,72){\line(0,-1){24}}
\put(278,72){\line(0,-1){24}}
\put(254,60){\line(1,0){36}}
\put(266,24){\makebox[0pt][l]{\framebox(24,12)}}
\put(278,24){\line(0,1){12}}
\thinlines
\put(236,90){\circle*{4}}
\put(236,90){\line(1,0){90}}
\put(236,90){\line(0,-1){90}} 

\put(248,78){\circle*{4}}
\put(248,78){\line(1,0){78}}
\put(248,78){\line(0,-1){78}} 

\put(296,66){\circle*{4}}
\put(296,66){\line(1,0){30}}
\put(296,66){\line(0,-1){66}} 

\put(308,54){\circle*{4}}
\put(308,54){\line(1,0){18}}
\put(308,54){\line(0,-1){54}} 

\put(260,42){\circle*{4}}
\put(260,42){\line(1,0){66}}
\put(260,42){\line(0,-1){42}} 

\put(320,30){\circle*{4}}
\put(320,30){\line(1,0){6}}
\put(320,30){\line(0,-1){30}}

\put(272,18){\circle*{4}}
\put(272,18){\line(1,0){54}}
\put(272,18){\line(0,-1){18}} 

\put(284,6){\circle*{4}}
\put(284,6){\line(1,0){42}}
\put(284,6){\line(0,-1){6}} 
\end{picture}
\]
The essential set boxes of $u(\tau)$ all lie in column $q=3$ while the essential set boxes
of $v(\tau)$ lie in column $p=5$, in agreement with Lemma~\ref{lemma:isvex}.\qed
\end{Example}

We now define {\bf pipe diagrams} associated to  $u(\gamma)$ for non-crossing $\gamma$. (The nomenclature
alludes to the ``pipe dreams'' terminology 
of \cite{Knutson.Miller:annals}.)
To start, replace each box of $D(u(\gamma))$ by a $+$.  The result is one of the
pipe diagrams. All other pipe diagrams are obtained from this first one by 
iterating the use of the local operation
\begin{equation}
\label{eqn:transformsconj}
\begin{matrix}
\cdot & \cdot\\
\cdot & +
\end{matrix} \ \mapsto \
\begin{matrix}
+ & \cdot\\
\cdot & \cdot
\end{matrix}
\end{equation}
with the additional restriction that no $+$'s appear in columns $q+1,q+2,\ldots,n$. The collection of all such pipe diagrams is denoted
${\tt Pipe}(u(\gamma))$. We define ${\tt Pipe}(v(\gamma))$ in the same way but using $D(v(\gamma))$ and requiring that there
are no $+$'s in columns $p+1,p+2,\ldots,n$. In addition, given any configuration ${\mathcal P}$ of $+$'s in the $n\times n$ grid
define
\[{\tt wt}^{(\beta)}({\mathcal P})=\prod_{\tiny{\mbox{$+$ in position $(i,j)$}}} x_i-y_j+\beta x_i y_j.\]

	We now explain why the initial conditions (\ref{eqn:matchlessdef}) 
defining $\Upsilon_{\tau}(X)$ agree with the ones from Section~1. Actually,
we have an extension. For $\gamma$ non-crossing, let 
$\tau^-$ be the matchless clan obtained by replacing each
left end of an arc by $-$ and any right end of an arc by $+$. Also,
let $\tau^+$ be the matchless clan obtained by replacing each
left end of an arc by $+$ and each right end of an arc by $-$. Define 
$\lambda(\gamma)$ to be $\lambda(\tau^-)$, and $\lambda(\widehat\gamma)$ to be
$\lambda(\widehat\tau^+)$, in the notation of the introduction.  Define also flaggings ${\vec f}(\gamma)$
and ${\vec f}(\widehat \gamma)$ to be ${\vec f}(\tau^-)$ and ${\vec f}(\widehat\tau^+)$, respectively.
The following result is straightforward from the results of \cite[Section~5]{KMY} (see specifically
Theorem~5.8) and the definitions of $u(\gamma)$, $v(\gamma)$, $\lambda(\gamma)$, and $\lambda(\widehat\gamma)$:
\begin{Proposition}
\label{prop:KMY}
For non-crossing $\gamma$ we have
\[
{\mathfrak S}^{(\beta)}_{u(\gamma)}(X;Y) =  \sum_{{\mathcal P}\in {\tt Pipe}(u(\gamma))} {\tt wt}^{(\beta)}({\mathcal P})
\mbox{ \ and \ }
{\mathfrak S}^{(\beta)}_{v(\gamma)}(X;Y) =  \sum_{{\mathcal P}\in {\tt Pipe}(v(\gamma))} {\tt wt}^{(\beta)}({\mathcal P}).\]
There is a (weight preserving) bijection between ${\tt Pipe}(u(\gamma))$ and semistandard set-valued Young tableaux of shape
$\lambda(\gamma)$ with flagging ${\vec f}(\gamma)$. The same holds for ${\tt Pipe}(v(\gamma))$
and semistandard set-valued Young tableaux of shape
$\lambda(\widehat\gamma)$ with flagging ${\vec f}(\widehat\gamma)$. In particular,
\[{\mathfrak S}_{u(\gamma)}(X)=s_{\lambda(\gamma),{\vec f}(\gamma)}(X) \mbox{ \ \ and  \ \ }
{\mathfrak S}_{v(\gamma)}(X)=s_{\lambda(\widehat\gamma),{\vec f}(\widehat\gamma)}(X).\]
\end{Proposition}


\begin{Proposition}
\label{prop:Snonly}
Suppose $\gamma$ is non-crossing and 
\[{\mathfrak S}^{(\beta)}_{u(\gamma)}(X;y_n,y_{n-1},\ldots,y_1){\mathfrak S}^{(\beta)}_{v(\gamma)}(X;Y)=\sum_{\kappa\in {\mathbb Z}_{\geq 0}^{\infty}}
c^{(\beta)}_{\kappa}(Y){\bf x}^{\kappa},\]
where ${\bf x}^{\kappa}=x_1^{\kappa_1}x_2^{\kappa_2}\cdots$
 and $c^{(\beta)}_{\kappa}(Y)\in {\mathbb Z}[\beta][Y]$. Then $c^{(\beta)}_{\kappa}(Y)=0$ unless $\kappa\leq (n-1,n-2,\ldots,3,2,1,0,0,0,\ldots )$
(component-wise comparison).
\end{Proposition}
\begin{proof}
Let us first show:

\begin{Claim}
\label{claim:staircase}
If ${\bf x}^{\kappa}$ appears in
${\mathfrak S}_{u(\gamma)}(X){\mathfrak S}_{v(\gamma)}(X)$ then $\kappa\subseteq (n-1,n-2,\ldots,2,1,0)\in {\mathbb Z}_{\geq 0}^n$. 
\end{Claim}
\begin{proof}
Suppose ${\bf x}^\kappa=\cdots x_i^m \cdots$. Let $\omega$ be
the width of the first non-empty row of $D(u(\gamma))$ 
that occurs in some row $s\geq i$ of $n\times n$.
Let $\omega'$ be the width of the first
nonempty row of $D(v(\gamma))$ that occurs in some row $t\geq i$
of $n\times n$. It is easy to see from the definitions that
\[m\leq \omega+\omega'.\]
We may assume without loss that $s$ and $t$ exist and also $t\geq s$  
(the alternate cases are proved similarly).

Let $A$ be the number of $-$'s or left ends of an arc occuring in the
leftmost $s$ positions of $\gamma$. Let $B$ be the number
of $+$'s or left ends of an arc occuring in the leftmost $t$ positions of
$\gamma$. Now
\[\omega=q-A \mbox{ \ \ \ and  \ \ \ } \omega'=p-B.\]
Since 
\[\omega+\omega'=p+q-A-B\]
it suffices to show $A+B\geq i$. Now, because in any left initial
segment of $\gamma$, the number of right ends of an arc is at most the number of left ends of an arc, we have:
\begin{eqnarray}\nonumber
A+B & \geq & A+\#\{\mbox{$+$ or left end of an arc in first $s$ 
								positions of $\gamma$}\}\\ \nonumber
        &\geq & A+\#\{\mbox{$+$ or right end of an arc in first $s$ 
								positions of $\gamma$}\}\\ \nonumber
       & = & s \ \ \geq i,\nonumber
\end{eqnarray} 
as desired.\end{proof}

\excise{We argue by double induction on $(p,q)$. First consider the case that $\gamma$ ends with a $+$.
Thus $u(\gamma)$ ends with $p+q$ and $v(\gamma)$ ends with $p$. Now
consider the clan $\gamma'$ obtained by removing that $+$. Thus $u(\gamma')$ is the same as $u(\gamma)$ with that
final $p+q$ removed; in particular, the diagrams of the two permutations is the same (up to an inclusion of an empty
row). Also, $v(\gamma')$ can be described as the unique permutation in $S_{p+q-1}$ which is obtained by removing
all boxes in column $p$. Here we have just used the fact that both $u(\gamma)$ and $v(\gamma)$ are inverse to a Grassmannian,
and in particular, all essential set boxes of $v(\gamma)$ lie in column $p$.
By induction any monomial ${\bf x}^{\kappa'}$ appearing in
${\mathfrak S}_{u(\gamma')}(X)\cdot {\mathfrak S}_{v(\gamma')}(X)$ satisfies
$\kappa'\subseteq (n-2,n-3,\ldots,3,2,1,0)\in {\mathbb Z}_{\geq 0}^{n-1}$. It then follows (using the ``$+$-diagram'' description
of flagged Schur polynomials) that the monomial ${\bf x}^{\kappa}$ is, at worst, equal to
${\bf x}^{\kappa'}\cdot x_1 x_2\cdots x_{n-1}$ and hence $\kappa\subseteq (n-1,n-2,\ldots,3,2,1,0)\subset {\mathbb Z}_{\geq 0}^n$,
as desired.

The argument in the case that $\gamma$ ends with $-$ is similar.}

Suppose the proposition is not true and 
there are set-valued tableaux $T$ and $U$ that contribute to ${\mathfrak S}^{(\beta)}_{u(\gamma)}(X;y_n,y_{n-1},\ldots,y_1)$ and ${\mathfrak S}^{(\beta)}_{v(\gamma)}(X;Y)$ respectively
(under the bijection of Proposition~\ref{prop:KMY}) such that 
the number of $i$'s in $T$ and $U$ combined strictly exceeds $n-i$, for
some $i$.
Now let $T'$ be the ordinary tableau that picks each of those $i$'s
as the representative of its box and picks any entry from the
remaining boxes. Since $T$ is semistandard, $T'$ is semistandard as well and contributes to ${\mathfrak S}_{u(\gamma)}(X)$.
Similarly, define $U'$, contributing to ${\mathfrak S}_{v(\gamma)}(X)$.
Then in ${\mathfrak S}_{u(\gamma)}(X){\mathfrak S}_{v(\gamma)}(X)$ the monomial
${\bf x}^{T'}{\bf x}^{U'}$ appears, contradicting Claim~\ref{claim:staircase}.
\end{proof}

It is well known (see, e.g., \cite[Proposition~2.5.4]{Manivel}) that the single Schubert polynomials $\{{\mathfrak S}_w(X): w\in S_{n}\}$ form a 
${\mathbb Z}$-linear basis of the vector space $\Gamma(X)$ of polynomials in $X$ using
only monomials ${\bf x}^{\kappa}$ where $\kappa\leq (n-1,n-2,\ldots,3,2,1)$. Now,
${\mathfrak G}_w(X)$ has the same lead term as ${\mathfrak S}_w(X)$ under the reverse
lexicographic order, namely ${\bf x}^{{\tt code}(w)}$. In addition, it is known (from
\cite{Fomin.Kirillov:groth}) that ${\mathfrak G}_{w}(X)\in \Gamma(X)$. Thus $\{{\mathfrak G}_w(X):w\in S_n\}$
also forms a basis of $\Gamma(X)$. Similarly, $\{{\mathfrak S}^{(\beta)}_w(X;Y): w\in S_n\}$ is a
${\mathbb Z}[\beta][Y]$-module basis of ${\mathbb Z}[\beta][Y]\otimes_{\mathbb Z} \Gamma(X)$. This is since ${\mathfrak S}^{(\beta)}_w(X;Y)$
also has leading term of ${\bf x}^{{\tt code}(w)}$ and if any term $c^{(\beta)}_{\kappa}(Y){\bf x}^{\kappa}$ is any
monomial then $\kappa\leq (n-1,n-2,\ldots,2,1,0)$.

Therefore, by Proposition~\ref{prop:Snonly}, when $\gamma$ is matchless
\[\Upsilon^{(\beta)}_{\gamma}(X;Y)=\sum_{w\in S_n} c^{(\beta)}_{\gamma,w}(Y){\mathfrak S}^{(\beta)}_{w}(X;Y).\]
Since $\partial^{(\beta)}_i$ sends $\beta$-Schubert polynomials to $\beta$-Schubert polynomials (or zero), such an expression where the summation is over $S_n$
holds for all clans.

	Given a clan $\gamma$ let $-\gamma$ be the clan where the $+$'s of $\gamma$ are replaced by $-$'s
and the $-$'s are replaced by $+$ (the arcs remain as is). We record the following property:

\begin{Proposition}[$\gamma\leftrightarrow -\gamma$ symmetry] Let $\gamma\in {\tt Clan}_{p,q}$. Then
\[\Upsilon_{-\gamma}^{(\beta)}(X;Y)=\Upsilon_{\gamma}^{(\beta)}(X;y_n,y_{n-1},\ldots,y_2,y_1).\]
\end{Proposition}
\begin{proof}
Let $\tau$ be a matchless clan such that 
\[\Upsilon_{\gamma}^{(\beta)}(X;Y)=\partial_{i_m}^{(\beta)}\cdots \partial_{i_1}^{(\beta)} 
\Upsilon_{\tau}^{(\beta)}(X;Y),\]
for some chain in weak Bruhat order from $\tau$ to $\gamma$ defined by $i_1,\ldots,i_m$.
Now we are done since the same sequence defines a chain from $-\tau$ to $-\gamma$ and because
the proposition is clear from the definitions for matchless $\tau$.
\end{proof}

In 
the ordinary cohomology, there is a further sense in which the choice of $\Upsilon_{\gamma}$ is simple.
Consider the degree lexicographic term order on polynomials in ${\mathbb Q}[x_1,\ldots,x_n]$. The \emph{Gr\"{o}bner normal form} is a distinguished representative of any coset modulo
$I^{S_n}$. The Schubert polynomials ${\mathfrak S}_w$ for $w\in S_n$ are the normal forms for their cosets; this is a fact due to \cite[Section~12.1]{FGP}. Thus
any linear combination of these Schubert polynomials is the normal form for its coset modulo $I^{S_n}$. Concluding:

\begin{Proposition}[Gr\"{o}bner normal form property]
\label{cor:normalform}
$\Upsilon_{\gamma}(X)$ is the Gr\"{o}bner normal form representative for the class of $[Y_{\gamma}]$ under the degree lexicographic term order. 
In other words, it is 
the unique representative that is a linear combination of $\{{\mathfrak S}_w:w\in S_n\}$.
\end{Proposition}

\subsection{Representatives in the Borel models}
We first explain our proof for equivariant cohomology (the argument in equivariant $K$-theory is completely
analogous). Let $T\subset GL_p\times GL_q$
be the torus of invertible diagonal matrices. Since each $Y_{\gamma}$ is $T$-stable, it admits a class $[Y_{\gamma}]_T\in H^\star_T(GL_n/B)$, a module over $H^\star_T(pt)\cong {\mathbb Z}[y_1,\ldots,y_n]$. The Borel-type model is
\begin{equation}
\label{eqn:equivBorel}
H^\star_T(GL_n/B)\cong {\mathbb Q}[X;Y]/J,
\end{equation}
where $J$ is the ideal generated by $e_i(X)-e_i(Y)$ and $e_i(X)$ is the elementary symmetric function in $X$, etc. 

\begin{Theorem}
\label{thm:equivver}
$\Upsilon_{\gamma}(X;Y)$ is well-defined and represents the coset of $[Y_{\gamma}]_T$ under (\ref{eqn:equivBorel}).
\end{Theorem}
(The forgetful map from $H^\star_T(GL_n/B)\twoheadrightarrow H^{\star}(GL_n/B)$ in this context amounts to setting each
$y_i=0$ and sends $[Y_{\gamma}]_T$ to $[Y_{\gamma}]$. Thus Theorem~\ref{thm:intro} follows from Theorem~\ref{thm:equivver}
since the forgetful maps and the Borel isomorphisms commute.)

The following is essentially standard. We include a proof for sake of completeness.

\begin{Proposition}
\label{prop:tworeps}
Suppose $f_1(X;Y)$ and $f_2(X;Y)$ are representatives of $[Y_{\gamma}]_T$ such that
\[f_1(X;Y)=\sum_{w\in S_n} a_w(Y){\mathfrak S}_w(X;Y) \mbox{ \ \  and \ \ }
f_2(X;Y)=\sum_{w\in S_n} b_w(Y){\mathfrak S}_w(X;Y).\]
Then $f_1(X;Y)=f_2(X;Y)$.
\end{Proposition}
\begin{proof}
We need that $a_w(Y)=b_w(Y)$ for all $w \in S_n$. 

Since $f_1$ and $f_2$ are equivariant cohomology class representatives of $[Y_\gamma]$ any substitution of $X$ by
a permutation $Y_{\sigma}=(y_{\sigma(1)},\ldots,y_{\sigma(n)})$ gives $f_1(Y_{\sigma};Y)=f_2(Y_{\sigma};Y)$ (this is where
we need that $\sigma\in S_n$).  This follows from the localization theorem for equivariant cohomology, combined with the fact 
that restriction to the $T$-fixed point $\sigma$ is given by 
\[ [Y_{\gamma}]_T|_{\sigma}= f_1(Y_{\sigma}; Y)=f_2(Y_{\sigma};Y). \]
These are standard facts, but the reader seeking a reference may consult \cite[Section~1.2]{Wyser-13b} for an expository treatment.  Also, \begin{equation}
\label{eqn:vanishing}
{\mathfrak S}_w(Y_\sigma;Y)=0 \mbox{\ if $\sigma\not\geq w$ in strong Bruhat order.}
\end{equation} 
Now, pick any linear extension
\[\pi^{(1)}=id,\pi^{(2)},\ldots,\pi^{(n!)}=w_0\]
of Bruhat order. Hence
\[a_{\pi^{(1)}}(Y){\mathfrak S}_{\pi^{(1)}}(Y_{\pi^{(1)}};Y)=f_1(Y_{\pi^{(1)}};Y)=f_2(Y_{\pi^{(1)}};Y)=
b_{\pi^{(1)}}(Y){\mathfrak S}_{\pi^{(1)}}(Y_{\pi^{(1)}};Y).\]
Since ${\mathfrak S}_{w}(Y_{w};Y)\neq 0$, dividing we conclude $a_{\pi^{(1)}}(Y)=b_{\pi^{(1)}}(Y)$.

Now set
\[f'_1(X;Y)=f_1(X;Y)-a_{\pi^{(1)}}(Y){\mathfrak S}_{\pi^{(1)}}(X;Y),\]
and
\[f'_2(X;Y)=f_2(X;Y)-a_{\pi^{(1)}}(Y){\mathfrak S}_{\pi^{(1)}}(X;Y).\]
Thus 
\[a_{\pi^{(2)}}(Y){\mathfrak S}_{\pi^{(2)}}(Y_{\pi^{(2)}};Y)=f'_1(Y_{\pi^{(2)}};Y)=f'_2(Y_{\pi^{(2)}};Y)=
b_{\pi^{(2)}}(Y){\mathfrak S}_{\pi^{(2)}}(Y_{\pi^{(2)}};Y),\]
and so $a_{\pi^{(2)}}(Y)=b_{\pi^{(2)}}(Y)$. 

Repeating, set
\[f''_1(X;Y)=f'_1(X;Y)-a_{\pi^{(2)}}(Y){\mathfrak S}_{\pi^{(2)}}(X;Y),\]
and
\[f''_2(X;Y)=f'_2(X;Y)-a_{\pi^{(2)}}(Y){\mathfrak S}_{\pi^{(2)}}(X;Y).\]
In this manner, we conclude all $n!$ desired equalities.
\end{proof}

We will establish the assumption of the following claim at the end of this section, and in a different way, in the next section.
\begin{Claim}
Assuming $\Upsilon_{\tau}(X;Y)$ represents $[Y_{\tau}]_T$ when $\tau$ is matchless, then
$\{\Upsilon_{\gamma}(X;Y)\}$ is self-consistent.
\end{Claim}
\begin{proof}
Pick a (non-matchless) clan $\\gamma$ and suppose there are two matchless clans $\tau_1$ and $\tau_2$
(possibly with $\tau_1=\tau_2$) such that
\[[Y_{\gamma}]_T=\partial_{i_m}\cdots\partial_{i_1}[Y_{\tau_1}]_T \mbox{\ \ and \ \ }
[Y_{\gamma}]_T=\partial_{j_m}\cdots\partial_{j_1}[Y_{\tau_2}]_T;\]
where we have mildly abused $\partial_i$ to mean the geometrically defined (equivariant)
push-pull operator on classes. We need to establish the \emph{polynomial} equality:
\[\partial_{i_m}\cdots\partial_{i_1}\Upsilon_{\tau_1}(X;Y)=
\partial_{j_m}\cdots\partial_{j_1}\Upsilon_{\tau_2}(X;Y).\]
Since we know $\Upsilon_{\tau_1}(X;Y)$ and $\Upsilon_{\tau_2}(X;Y)$ expand into double Schubert
polynomials (from $S_n$), the claim follows from Proposition~\ref{prop:tworeps}.
\end{proof}

\excise{Since we will establish that $\Upsilon_{\gamma}$ and $\Upsilon_{\delta}$
are polynomial representatives for $[Y_{\gamma}]$ and $[Y_{\delta}]$ respectively,
we automatically know the weaker fact that
\[\partial_{i_m}\cdots\partial_{i_1}\Upsilon_{\gamma}=
\partial_{j_m}\cdots\partial_{j_1}\Upsilon_{\delta} \mbox{\ \ \  (mod $I^{S_n}$)},\]
i.e., that $\Upsilon_{\alpha}^{\circ}:=\partial_{i_m}\cdots\partial_{i_1}\Upsilon_{\gamma}$ and $\Upsilon'_{\alpha}:=\partial_{j_m}\cdots\partial_{j_1}\Upsilon_{\delta}$ are both polynomial representatives of the same coset.

We use some facts from the theory of Gr\"{o}bner bases.
Consider the degree lexicographic term order on polynomials in ${\mathbb Q}[x_1,\ldots,x_n]$. With this we have the \emph{normal form} with respect to this
order. The normal form is a distinguished representative of any coset modulo
$I^{S_n}$. It is known that the Schubert polynomials ${\mathfrak S}_w$ for $w\in S_n$ are the normal forms for their cosets. It follows therefore that
any linear combination of these Schubert polynomials is the normal form for
its coset modulo $I^{S_n}$.

To conclude, by (\ref{eqn:isS_n}) we have that both
$\Upsilon_{\alpha}^{\circ}$ and $\Upsilon'_{\alpha}$ are linear combinations
of Schubert polynomials from $S_n$. Thus they are the normal forms for their
cosets. However, they are both representatives of the same coset, hence
$\Upsilon_{\alpha}^{\circ}=\Upsilon'_{\alpha}$, as desired.}

Following \cite[Section 2.3]{Knutson.Miller:annals}, the 
$K$-cohomology ring $K^{\circ}(GL_n/B)$ has the presentation
\[K^{\circ}(GL_n/B)\cong {\mathbb Z}[x_1,\ldots,x_n]/K\]
where $K$ is the ideal generated by $e_d(x_1,\ldots,x_n)-{n\choose d}$ for $d\leq n$; here $e_d(x_1,\ldots,x_n)$ is the
elementary symmetric function of degree $d$. Next, following \cite{Fulton.Lascoux}
if we let $K^\circ_T(GL_n/B)$ denote the $T$-equivariant $K$-theory ring of $GL_n/B$ then
\[K^\circ_T(GL_n/B)=K^\circ_T(pt)[x_1,\ldots,x_n]/J\cong {\mathbb Z}[y_1^{\pm 1},\ldots y_n^{\pm 1}][x_1,\ldots,x_n]/J,\]
where $J$ is as in (\ref{eqn:equivBorel}).
In these senses, one can speak of a (Laurent) polynomial ``representing'' the class of a structure sheaf of a ($T$-stable) variety in $GL_n/B$.

\begin{Theorem}
\label{thm:K}
The families $\{\Upsilon^K_{\gamma}(X)\}$ and $\{\Upsilon^K_{\gamma}(X;Y)\}$
are well-defined. Moreover, 
$\Upsilon^K_{\gamma}(X)$ represents $[{\mathcal O}_{Y_{\gamma}}]\in K^{\circ}(GL_n/B)$ and
$\Upsilon^K_{\gamma}(X;Y)$ represents $[{\mathcal O}_{Y_{\gamma}}]_T \in K^{\circ}_T(GL_n/B)$.
\end{Theorem}

The proof is exactly the same 
as in equivariant cohomology, except one must use equivariant $K$-theory localization.  This requires the now standard fact that, in equivariant $K$-theory, $[{\mathcal O}_{Y_{\gamma}}]_T|_{\sigma} = f(Y_{\sigma}; Y)$
when $f(X;Y)$ is a representative of $[{\mathcal O}_{Y_{\gamma}}]_T$ in the Borel model.  We are unaware of a specific 
reference for it in the literature, so we remark here that the argument of \cite[Proposition 1.3]{Wyser-13b} can be modified to apply to $K$-theory simply by replacing the first Chern classes of the tautological line bundles 
by (the classes of) the bundles themselves. A recent reference for equivariant
localization in $K$-theory is \cite{Harada}. The analogues of the vanishing conditions on Schubert classes (\ref{eqn:vanishing})
also hold. One also needs the following, which should also be straightforward to experts,
but for which we are also not aware of a proof in the literature:

\begin{Proposition}
The isobaric divided difference operator $\pi_i=\partial^{(1)}_i$ takes a representative of the class of 
$Y_{\gamma}$ to one for $Y_{s_i\gamma}$ in (equivariant) $K$-theory of $GL_n/B$.
\end{Proposition}
\begin{proof}
Let $Y = Y_{\gamma}$, and $Y' = Y_{s_i \gamma}$.  First, recall that all orbit closures for this case are multiplicity-free, meaning that their cycle classes in the Chow ring can be expressed in the Schubert basis with all coefficients $0$ or $1$.  This is noted in \cite{Brion} and further 
elaborated upon in \cite{Wyser12}.  Thus by \cite[Theorem 6]{Brion}, $Y$ has rational singularities.  Let $\pi: G/B \rightarrow G/P_{\ga_i}$ be the natural projection where $P_{\alpha_i}$ is the minimal parabolic associated
to $\alpha_i$.  Since $Y' = \pi^{-1}(\pi(Y))$ is a ${\mathbb P}^1$-bundle over $\pi(Y)$, and since $Y'$ has rational singularities (being another multiplicity-free $K$-orbit closure), $\pi(Y)$ has rational singularities as well.

Now we note that the proof of \cite[Lemma 4.12]{KK} or
\cite[Theorem~3]{Fulton.Lascoux}, given there for (equivariant) $K$-classes of Schubert varieties, applies to the case at hand.  
\end{proof}

\noindent
\emph{Conclusion of proof of Theorems~\ref{thm:intro},~\ref{thm:equivver}
and~\ref{thm:K}:} It remains to show that the proposed representatives are indeed representatives
for the closed orbits. This follows from three facts. First, by \cite{Wyser12}, when $\gamma$ is non-crossing, 
\[Y_{\gamma}=X_{v(\gamma)}^{w_0u(\gamma)}:={\overline{B_{-}v(\gamma)B/B}}\cap {\overline{Bw_0 u(\gamma)B/B}}.\]
Second, in the case of equivariant $K$-theory, the representative of the Schubert variety $X_{w}$ is ${\mathfrak G}_w(X;Y)$; this is proved
in \cite[Theorem~3]{Fulton.Lascoux}. It also follows from \emph{loc. cit.} that 
 ${\mathfrak G}_{w_0w}(X;y_n,y_{n-1},\ldots,y_1)$ represents the opposite Schubert variety
$X^w={\overline{BwB/B}}$. Similarly, it is known that ${\mathfrak S}_w(X)$, ${\mathfrak S}_w(X;Y)$
and ${\mathfrak G}_w(X)$ represent the Schubert classes in the corresponding cohomology theories, and ${\mathfrak S}_{w_0w}(X)$,
${\mathfrak S}_{w_0w}(X;y_n,\ldots,y_1)$ and ${\mathfrak G}_{w_0w}(X)$ represent the opposite Schubert classes.
Finally, $[X_u^v]=[X_u][X^v]$ (interpreted in any of the cohomology theories
we are using). \qed

\begin{Remark}[Positivity]
The argument of the introduction that $\Upsilon_{\gamma}(X)\in {\mathbb Z}_{\geq 0}[x_1,\ldots,x_n]$ extends to prove appropriate notions of ``positivity'' for each of the given representatives associated to the other three cohomology theories. This is since in each
case there is an available notion of positivity of Schubert calculus. See \cite{AGM} and the references therein.\qed
\end{Remark}

\begin{Remark}
Consider $X_{2143}^{3412}$. It is true that
${\mathfrak S}_{2143}(x_1,x_2,x_3,x_4)^2=x_1^4+\ldots$ represents the class of the Richardson variety. However, this polynomial is not the normal form representative of its coset because it involves $x_1^4$ (cf. Proposition~\ref{cor:normalform} and see also \cite{Lenart.Sottile}). 
This emphasizes the role of Proposition~\ref{prop:Snonly} in our proofs.\qed 
\end{Remark}

\begin{Remark}
Our arguments show that if any collection of varieties in $GL_n/B$ have their classes related by 
(isobaric) divided difference operators then any choice of polynomial representatives
for their minimal elements that expand into Schubert polynomials
from $S_n$ gives a self-consistent family of representatives. In particular 
this can also be applied to the
cases where $(G,K)=(GL_{2n}, Sp_{2n})$ and $(G,K)=(GL_n,O_n)$; cf. \cite{Wyser-13b} and \cite{WyYo2}.
\end{Remark}

\section{The $K$-orbit determinantal ideal}

\subsection{Geometric naturality of $\Upsilon_{\gamma}^{(\beta)}$} 
The {\bf $K$-orbit determinantal ideal} $I_\gamma$ is defined as follows.
Fix $\gamma\in {\tt Clan}_{p,q}$.  For $i=1,\hdots,n$, let:
\begin{itemize}
	\item $\gamma(i;+)$ = the total number of $+$'s and matchings in the first $i$ vertices, and
	\item $\gamma(i;-)$ = the total number of $-$'s and matchings in the first $i$ vertices.
\end{itemize}

For $1 \leq i < j \leq n$, define
\begin{itemize}
	\item $\gamma(i;j) = \#\{k \in [1,i] \mid \mbox{$k$ and $l$ are matched} \text{ and } l > j\}$.
\end{itemize}

Let $R_{+}(\gamma)$ be the vector with $i$-th entry equal to $i+1-\gamma(i;+)$
and $R_{-}(\gamma)$ be the vector with $i$-th entry $i+1-\gamma(i;-)$. Also,
let $W(\gamma)$ be the $n\times n$ matrix whose $(i,j)$-th entry
is $j+\gamma(i;j)+1$ if $i<j$ and is zero otherwise.

	Identify ${\rm Fun}({\rm Mat}_{n\times n})$ with ${\mathbb C}[z_{i,j}]$ where $z_{i,j}$ is the coordinate
function of matrix coordinate $(i,j)$. Let $M_n$ be the generic $n\times n$ matrix with entry $z_{i,j}$. Now define
$I_\gamma$ to have the following generators:

\begin{enumerate}
	\item[(i)] For each $i = 1,\hdots,n$, the minors of size 
$R_{+}(\gamma)_i$ of the lower-left $q \times i$ submatrix of $M_n$.
	\item[(ii)] For each $i =  1,\hdots,n$, the minors of size 
$R_{-}(\gamma)_i$ of the upper-left $p \times i$ submatrix of $M_n$.
	\item[(iii)] For each $1 \leq i < j \leq n$, the minors of size 
$W(\gamma)_{ij}$ of the following $n \times (i+j)$ matrix $P_{i,j}$:   The upper-left $p \times i$ block coincides with the upper-left $p \times i$ block of $M_n$, the lower-left $q \times i$ block is zero, and the last $j$ columns coincide with the first $j$ columns of $M_n$.
\end{enumerate}

\begin{Example}
Let $\gamma=
\begin{picture}(30,10)
\put(0,0){\epsfig{file=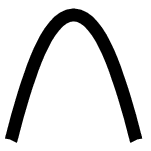, height=.4cm}}
\put(11,1){$+-$}
\end{picture}$.
 Then
\[
\begin{array}{ccc}
\gamma_+ & = & (0,1,2,2)\\
R_{+}(\gamma) & = & (2,{\underline 2},{\underline 2},3)
\end{array}
, \ 
\begin{array}{ccc}
\gamma_{-} & = & (0,1,1,2)\\
R_{-}(\gamma) & = & (2,{\underline 2},3,3)
\end{array}, \ 
W(\gamma)=\begin{tabular}{c|cccc}
$i$/$j$ & $1$ & $2$ & $3$ & $4$ \\
\hline
$1$ & \ & \underline{$3$} & \underline{$4$} & $5$\\
$2$ & \ & \ & {\underline{$4$}} & $5$\\
$3$ & \ & \ & \ & $5$ \\
$4$ & \ & \ & \ & \ \\
\end{tabular}\]
Not all rank conditions give rise to non-trivial minors; we have underlined those that do. Specifically $R_+$ demands that the $2\times 2$ minors
of the southwest $2\times 3$ submatrix of
\[M_4=\left(\begin{matrix}
z_{11} & z_{12} & z_{13} & z_{14}\\
z_{21} & z_{22} & z_{23} & z_{24}\\
z_{31} & z_{32} & z_{33} & z_{34}\\
z_{41} & z_{42} & z_{43} & z_{44}
\end{matrix}\right)\]
be among the generators. $R_{-}$ contributes the $2\times 2$ northwest minor
of this matrix. Here,
\[P_{1,2}=\left(\begin{matrix}
z_{11} & z_{11} & z_{12} \\
z_{21} & z_{21} & z_{22} \\
0 & z_{31} & z_{32} \\
0 & z_{41} & z_{42} 
\end{matrix}\right),
P_{1,3}=\left(\begin{matrix}
z_{11} & z_{11} & z_{12} & z_{13}\\
z_{21} & z_{21} & z_{22} & z_{23}\\
0 & z_{31} & z_{32} & z_{33}\\
0 & z_{41} & z_{42} & z_{43}
\end{matrix}\right),
P_{2,3}=\left(\begin{matrix}
z_{11} &  z_{12} & z_{11} & z_{12} & z_{13}\\
z_{21} & z_{22} & z_{21} & z_{22} & z_{23}\\
0 & 0 & z_{31} & z_{32} & z_{33}\\
0 & 0 & z_{41} & z_{42} & z_{43}
\end{matrix}\right).
\]
The conditions from $W(\gamma)$ say that we add the $3\times 3$ minors
of $P_{1,2}$ and the $4\times 4$ minors of $P_{1,3}$ and of $P_{2,3}$.

Actually, the rank conditions from $R_{+}$ and $R_{-}$
already imply the minors from $W(\gamma)$. This is true for all 
non-crossing $\gamma$, as explained below. 
\qed
\end{Example}

Our reference for combinatorial commutative algebra, specifically the notion of multigrading, multidegree and
$K$-polynomial is \cite[Chapter~8]{Miller.Sturmfels} as well as the connection to equivariant cohomology. For brevity, we refer the reader to that textbook for basic
definitions and notions.

Let $\prec_{p,q}$ be the lexicographic term order on monomials in $\{z_{i,j}\}$ that orders the variables by reading the
bottom $q$ rows, from left to right and from bottom to top, followed by the top $p$ rows from left to right and from top to bottom.
The $T\times T$ action on $M_n$ restricts to an action on $M_{\gamma}=\overline{\pi^{-1}(Y_{\gamma})}$. The associated grading associated to multidegrees
assigns the variable $z_{ij}$ the weight $x_j-y_i$. For $K$-polynomials, the grading assigns $z_{ij}$ the weight $1-\frac{x_j}{y_i}$. 

The following result explains the geometric naturality of our choices for representatives of the closed orbits.
It also applies more generally to orbit closures indexed by
non-crossing clans.

\begin{Theorem}
\label{claim:main}
Suppose $\gamma$ is non-crossing. Then $M_{\gamma}$ is scheme-theoretically cut out by $I_{\gamma}$. Also:
\begin{itemize}
\item[(I)] The defining equations of $I_{\gamma}$ 
form a Gr\"{o}bner basis with squarefree lead terms, with respect to the
term order $\prec_{p,q}$. 
\item[(II)] The
Gr\"{o}bner limit ${\rm init}_{\prec_{p,q}}(I_{\gamma})$ 
has a prime decomposition whose components are naturally indexed by pairs of semistandard tableaux $(T,U)$ 
where 
\begin{itemize}
\item[$\bullet$] $T$ is a flagged tableaux of shape $\lambda(\gamma)$ with flagging ${\vec f}(\lambda(\gamma))$; and
\item[$\bullet$] $U$ is a flagged tableaux of shape $\mu(\gamma)$ with flagging ${\vec f}(\mu(\gamma))$.
\end{itemize}
\item[(III)] ${\rm multidegree}_{{\mathbb Z}^{2n}}({\mathbb C}[Z]/I_{\gamma}) =  \Upsilon_{\gamma}(X;Y)$ and ${\mathcal K}_{{\mathbb Z}^{2n}}({\mathbb C}[Z]/I_{\gamma}) =  \Upsilon^K_{\gamma}(X;Y)$.
\end{itemize}
\end{Theorem}

\begin{Example}
Continuing our previous example, one checks that 
\begin{eqnarray}\nonumber 
{\rm init}_{\prec_{2,2}}(I_{\begin{picture}(30,10)
\put(0,0){\epsfig{file=matchingsSection4.eps, height=.4cm}}
\put(11,1){$+-$}
\end{picture}})& =& \langle z_{42}z_{33}, 
z_{41}z_{33}, z_{41}z_{32}, z_{11}z_{22}\rangle\\ \nonumber
\ & = & \langle z_{11}, z_{41},z_{42}\rangle \cap \langle z_{11},z_{41},z_{33}\rangle \cap \langle z_{11},z_{32}, z_{33}\rangle \cap \langle z_{22},z_{41},z_{42}\rangle \cap \\ \nonumber
\ & \ & \langle z_{22},z_{41},z_{33}\rangle \cap 
 \langle z_{22},z_{32},z_{33}\rangle. \nonumber
\end{eqnarray}
 Now consider the pipe diagrams associated to each prime component
of ${\rm init}_{\prec_{2,2}}(I_{\begin{picture}(30,10)
\put(0,0){\epsfig{file=matchingsSection4.eps, height=.4cm}}
\put(11,1){$+-$}
\end{picture}})$: we define them to be obtained by placing a ``$+$'' in 
position $(i,j)$ if $z_{ij}$ appears in the component. These are respectively:
\[\left[\ \begin{matrix}
+ & . & . & . \\
. & . & . & . \\
\hline
. & . & . & . \\
+ & + & . & . \\
\end{matrix}\ \right], 
\left[\ \begin{matrix}
+ & . & . & . \\
. & . & . & . \\
\hline
. & . & + & . \\
+ & . & . & . \\
\end{matrix}\ \right], 
\left[\ \begin{matrix}
+ & . & . & . \\
. & . & . & . \\
\hline
. & + & + & . \\
. & . & . & . \\
\end{matrix}\ \right], 
\left[\ \begin{matrix}
. & . & . & . \\
. & + & . & . \\
\hline
. & . & . & . \\
+ & + & . & . \\
\end{matrix}\ \right], 
\]
\[\left[\ \begin{matrix}
. & . & . & . \\
. & + & . & . \\
\hline
. & . & + & . \\
+ & . & . & . \\
\end{matrix}\ \right], 
\left[\ \begin{matrix}
. & . & . & . \\
. & + & . & . \\
\hline
. & + & + & . \\
. & . & . & . \\
\end{matrix}\ \right].\]
To compute the ${\mathbb Z}^{2n}$ multidegree one uses additive grading that assigns $z_{i,j}$ the weight $x_j-y_i$.
Then
\begin{eqnarray}\nonumber
{\rm multidegree}_{{\mathbb Z}^{2n}}({\mathbb C}[Z]/I_{\gamma}) & = & (x_1-y_4)(x_2-y_4)\cdot  (x_1-y_1)
+ (x_1-y_4)(x_3-y_3)\cdot (x_1-y_1) \\ \nonumber
& & + (x_3-y_2)(x_3-y_3)\cdot (x_1-y_1) + (x_1-y_4)(x_1-y_3)\cdot (x_2-y_2) \\ \nonumber
& & + (x_1-y_4)(x_3-y_3)\cdot (x_2-y_2) + (x_2-y_3)(x_3-y_3)\cdot (x_2-y_2) \nonumber
\end{eqnarray}
In each term, we use ``$\cdot$'' to separate the factors coming from $+$'s below and above the horizontal line
of the corresponding pipe diagram. Factoring gives
\begin{eqnarray}\nonumber
& [(x_1-y_4)(x_2-y_4)+(x_1-y_4)(x_2-y_3)+(x_2-y_3)(x_3-y_3)]\cdot  [(x_1-y_1)+(x_2-y_2)]\\ \nonumber
= & s_{(1,1),(2,3)}(x_1,x_2,x_3,x_4;y_4,y_3,y_2,y_1)s_{(1,0),(2,4)}(x_1,x_2,x_3,x_4;y_1,y_2,y_3,y_4)\nonumber
\end{eqnarray}
in agreement with the theorem. One can also similarly verify the $K$-polynomial claim by computing the $K$-polynomial
of the simplicial complex associated to ${\rm init}_{\prec_{p,q}}I_{\gamma}$.\qed
\end{Example}

\noindent
\emph{Proof of Theorem~\ref{claim:main}:} We recall \cite[Theorem 2.5]{Wyser-13a}: Let
\[E_p={\rm span}\{\vec e_1,\vec e_2,\ldots,\vec e_p\} \mbox{\ and \ }  E^q={\rm span}\{\vec e_{p+1}, \vec e_{p+2},\ldots, \vec e_n\}, \] 
where $\vec e_i$ is the $i$-th standard basis vector of $\C^n$ and $\rho: \C^n \rightarrow E_p$ is the natural projection map. 
\begin{Theorem}
\label{thm:wyserbruhat}
$Y_{\gamma}$ is the set of flags $F_{\bullet}$ such that the following three conditions hold:
\begin{enumerate}
	\item $\dim(F_i \cap E_p) \geq \gamma(i;+)$ for all $i$;
	\item $\dim(F_i \cap E^q) \geq \gamma(i;-)$ for all $i$;
	\item $\dim(\rho(F_i) + F_j) \leq j + \gamma(i;j)$ for all $i<j$.
\end{enumerate}
\end{Theorem}

\noindent
Recall that $\pi:GL_n\to GL_n/B$ is the natural map. Consider the following diagram:
\[\begin{CD}
\pi^{-1}(Y_{\gamma}) @. \ \ \subset \ \ @. GL_n @. \subset {\rm Mat}_n\supset \overline{\pi^{-1}(Y_\gamma)}:=M_{\gamma}\subseteq V(I_\gamma)\\
                      @. @.        @VV{\pi}V\\
Y_{\gamma} @.\  \subset \ \ @.  GL_n/B @.
\end{CD}\]

\begin{Lemma}
\label{lemma:vanish}
Suppose $g\in GL_n$. Then $g\in \pi^{-1}(Y_{\gamma})$ if and only if $g$ vanishes on all generators (i), (ii) and (iii)
of $I_{\gamma}$.
\end{Lemma}
\begin{proof}
We make the usual identification of $gB\in GL_n/B$ with the flag 
\[F_{\bullet}:\langle {\vec 0}\rangle\subset F_1\subset F_2\subset \ldots \subset F_{n-1}\subset {\mathbb C}^n,\]
where $F_i$ is spanned by the leftmost $i$ columns of $g$.

Fix $F_{\bullet}\in Y_{\gamma}$. By Theorem~\ref{thm:wyserbruhat}, the conditions (1), (2) and (3) of that theorem
hold. We examine their implications on $g$:

(1) and (2): Consider the map $\phi:F_i\to E^q$ obtained by the projection of ${\vec v}\in F_i\subset {\mathbb C}^n$
onto ${\mathbb C}^n/E_p\cong E^q$. Since $\ker \phi = F_i\cap E_p$, by the rank-nullity theorem, (1) is equivalent to
\[{\rm rank} \ \phi = \dim F_i- \dim \ker \phi \leq i-\gamma(i,+).\]
Equivalently, the $g$ associated to $F_{\bullet}$ vanishes on the minors (i). Similarly, $F_{\bullet}$ satisfies 
(2) if and only if $g$ vanishes on the minors (ii).

(3): $\rho(F_i) + F_j$ is isomorphic to the 
 column space of the $n \times (i+j)$ matrix whose first $i$ columns coincide with the first $i$ columns of $g$, 
but with the lower-left $q \times i$ submatrix zeroed out, and whose next $j$ columns coincide 
with the first $j$ columns of $g$ (unaltered). Thus $g$ vanishes on the generators (iii) if and only if
$F_{\bullet}$ satisfies (3). 
\end{proof}

Since $I_{\gamma}$ vanishes on $\pi^{-1}(Y_{\gamma})$ we must have $M_{\gamma}:=\overline{\pi^{-1}(Y_{\gamma})}\subseteq V(I_{\gamma})$. We would know $M_{\gamma}=V(I_{\gamma})$ (as sets) if the latter
is shown to be irreducible.

Let ${\widetilde I}_{\gamma}$ be generated by the generators (i) and (ii).  We will need to recall the following well-known and easy fact about 
Gr\"{o}bner bases, stated in the specific form we need:

\begin{Lemma}
\label{lemma:easygrobner}
Let $A$ and $B$ be disjoint collections of commuting variables.
Suppose $f_1,\ldots, f_n$ is a Gr\"{o}bner basis
of ${\Bbbk}[A]$ with respect to a pure lexicographic term order $\prec_A$, and that $g_1,\ldots,g_m$ is a Gr\"{o}bner basis of
${\Bbbk}[B]$ with respect to a pure lexicographic term order $\prec_B$. Let $\prec_{A;B}$ be the pure lexicographic term order on
${\Bbbk}[A,B]$ extending $\prec_A$ and $\prec_B$ that favors $A$ over $B$. Then $G=\{f_1,\ldots,f_n, g_1,\ldots g_m\}$ is a Gr\"{o}bner basis with respect to $\prec_{A,B}$. 
\end{Lemma}
\begin{proof}
Indeed, if $S(f_i,g_j)$ is the $S$-polynomial then
\[S(f_i,g_j):={\tt LT}(g_j)f_i - {\tt LT}(f_i)g_j=-(g_j-{\tt LT}(g_j))f_i+(f_i-{\tt LT}(f_i))g_j.\]
Thus, using the multivariate division algorithm, dividing $S(f_i,g_j)$ by $G$ (listed in the order $f_i,g_j,\ldots$) gives remainder $0$. Now apply Buchberger's criterion \cite[Section~15.4]{Eisenbud}.
\end{proof}

\begin{Claim}
\label{claim:grobfortilde}
${\widetilde I}_{\gamma}$ is a prime ideal that scheme-theoretically cuts
out $\overline{\pi^{-1}(X_{u(\gamma)}^{v(\gamma)})}$. The generators form a Gr\"{o}bner basis (with squarefree lead terms)
with respect to $\prec_{p,q}$.
\end{Claim}
\begin{proof}
By definition, ${\widetilde I}_\gamma$ is the ideal sum of a Schubert determinantal ideal associated
to $v(\gamma)$ living in the
first $p$ rows with a Schubert determinantal ideal associated to $u(\gamma)$ living in the bottom $q$ rows. 
The generators for each of these is individually Gr\"{o}bner (with squarefree lead terms) for the term order given \cite[Theorem~3.8]{KMY}. Now the Gr\"{o}bner assertion
holds by Lemma~\ref{lemma:easygrobner}.

Since $\pi^{-1}(X_{u(\gamma)}^{v(\gamma)})$ clearly vanishes on ${\widetilde I}_{\gamma}$ we have 
$\overline{\pi^{-1}(X_{u(\gamma)}^{v(\gamma)})}\subseteq V({\widetilde I}_{\gamma})$ (and both zero sets
are of the same dimension).

Since its generators are squarefree and Gr\"{o}bner, by semicontinuity, ${\widetilde I}_{\gamma}$ is a radical ideal. 
On the other hand, $V({\widetilde I}_{\gamma})$ is clearly irreducible since it is the Cartesian product of two (irreducible)
matrix Schubert varieties. Hence by the Nullstellensatz, ${\widetilde I}_{\gamma}$ is prime and so 
$\overline{\pi^{-1}(X_{u(\gamma)}^{v(\gamma)})}= V({\widetilde I}_{\gamma})$ (scheme-theoretic equality). 
\end{proof}

Now we have 
\[\overline{\pi^{-1}(X_{u(\gamma)}^{v(\gamma)})}= V({\widetilde I}_{\gamma})\supseteq V(I_{\gamma})\supseteq M_{\gamma}.\]
However, by \cite{Wyser12} we know $Y_{\gamma}=X_{u(\gamma)}^{v(\gamma)}$ so $M_{\gamma}=\overline{\pi^{-1}(X_{u(\gamma)}^{v(\gamma)})}$
and hence $V({\widetilde I}_{\gamma})=V(I_{\gamma})$.
Furthermore, by the Nullstellensatz, $I_\gamma\subseteq {\widetilde I}_{\gamma}(=\sqrt {\widetilde I}_{\gamma})$. However, by definition $I_{\gamma}\supseteq 
{\widetilde I}_{\gamma}$ and hence $I_{\gamma}={\widetilde I}_{\gamma}$. Thus (I) now follows by Claim~\ref{claim:grobfortilde}
since the additional generators (with squarefree lead terms) that are in $I_{\gamma}$ but not ${\widetilde I}_{\gamma}$ do not affect
Gr\"{o}bnerness of the latter's generators, for general reasons.

Since $M_{\gamma}\subseteq V(I_\gamma)$ and now $I_{\gamma}={\widetilde I}_{\gamma}$ is prime, we must have $M_{\gamma}=V(I_{\gamma})$ and the first sentence of the theorem holds.

In view of the equality $I_{\gamma}={\widetilde I}_{\gamma}$, (II) is now easy from \cite[Section~4]{KMY} since the latter
is the ideal sum of two vexillary Schubert determinantal ideals. (Specifically note that our grading of $z_{ij}$ is transpose to the
convention used in \emph{loc. cit.})

Given (II), (III) follows from Proposition~\ref{prop:KMY} and the conclusion of our 
proof of the main theorems of Section~2. Alternatively, by the same line of reasoning as
\cite[Corollary~2.3.1]{Knutson.Miller:annals},  ${\rm multidegree}_{{\mathbb Z}^{2n}}({\mathbb C}[Z]/I_{\gamma})$ represents $[Y_{\gamma}]_T$. However, \emph{a priori} this representative is not the same as $\Upsilon_{\gamma}(X;Y)$.
That these are in fact equal follows from (II), Proposition~\ref{prop:Snonly} and Proposition~\ref{prop:tworeps}. The authors
of \emph{loc. cit.} in fact explain how their argument works in ordinary $K$-theory; cf. \cite[Remarks 2.3.3 and 2.3.4]{Knutson.Miller:annals}.
\qed

\subsection{Conjectures}\label{sec:conjectures}
Some of the assertions of Theorem~\ref{claim:main} seem to hold generally.

\begin{Conjecture}
\label{conj:A}
The generators of $I_{\gamma}$ are a Gr\"{o}bner basis with respect to some lexicographic ordering. In particular, $I_{\gamma}$ is a radical ideal. 
\end{Conjecture}

We emphasize that the term order needed generally depends on $\gamma$.  Conjecture \ref{conj:A} has been verified exhaustively for $p+q\leq 6$ as well as in enough cases for larger $p+q$ for us to be convinced. 

\begin{Example}
When $(p,q)=(1,2),(2,1)$, all $\gamma$ are non-crossing.
When $(p,q)=(2,2), (3,2)$, $\prec_{p,q}$ succeeds in making the
defining generators of $I_{\gamma}$ Gr\"{o}bner. 
This term order also succeeds for $(p,q)=(2,3)$ if one add some generators obtained by column operations on the $P_{i,j}$ matrices. 
The first interesting examples seem to be at $(p,q)=(3,3)$ where

\[
\begin{picture}(280,15)
\put(0,0){\epsfig{file=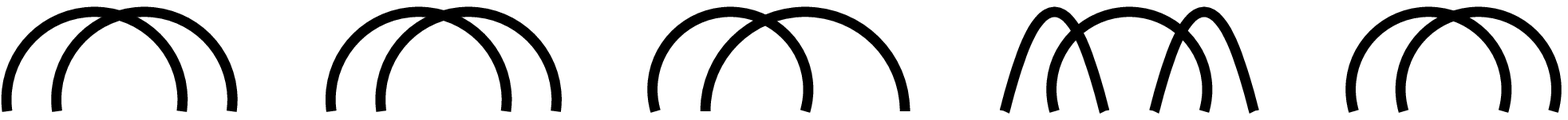, height=.8cm}}
\put(13,0){$-$}
\put(23,0){$+$}
\put(49,0){$,$}

\put(74,0){$+$}
\put(84,0){$-$}
\put(110,0){$,$}

\put(176,0){$,$}
\put(242,0){$,$}

\put(139,0){$+$}
\put(157,0){$-$}
\put(272,0){$+$}
\put(300,0){$-$}
\end{picture}
\]
are the instances where the defining generators (or the modification
alluded to above) are not Gr\"{o}bner with respect to $\prec_{p,q}$.\qed
\end{Example}


\begin{Conjecture}
\label{conj:B}
Theorem~\ref{claim:main}(III) holds for all $\gamma$. 
\end{Conjecture}
Equivalently, Conjecture~\ref{conj:B} claims that ${\rm multidegree}_{{\mathbb Z}^{2n}}({\mathbb C}[Z]/I_{\gamma})$ 
and ${\mathcal K}_{{\mathbb Z}^{2n}}({\mathbb C}[Z]/I_{\gamma})$ satisfy the divided difference and isobaric
divided difference recurrences.  This has been verified by exhaustive computer checks through $p+q=7$.  
%

We are not yet confident to assert that $I_\gamma$ is prime, although further discussion may appear elsewhere.

These algebraic problems are closely related to two 
combinatorial questions:

\begin{Question}
\label{question:C}
Give a manifestly nonnegative combinatorial rule for the expansion of $\Upsilon^{(\beta)}_{\gamma}(X;Y)$ into monomials in
$x_i-y_j+\beta x_i y_j$.
\end{Question}

\begin{Question}
\label{question:D}
Give a manifestly nonnegative combinatorial rule for the expansion of $\Upsilon^{(\beta)}_{\gamma}(X;Y)$ into 
${\mathfrak S}^{(\beta)}_{\gamma}(X;Y)$.
\end{Question}

Brion's formula \cite{Brion} states:
\[[Y_{\gamma}]=\sum_{w\in S_n} c_{\gamma,w}[X_w]\in H^{\star}(GL_n/B),\]
for explicit, combinatorially defined coefficients $c_{\gamma,w}\in\{0,1\}$. 
In view of Proposition~\ref{prop:Snonly}, this formula implies a solution to Question~\ref{question:D} when $\beta=0$ and each $y_i=0$,  by
using any monomial expansion formula (e.g., \cite{Fomin.Kirillov, BB}) for ${\mathfrak S}_{w}(X)$.

A result of A.~Knutson \cite[Theorem~3]{Knutson:multfree} shows how to obtain the $K$-theoretic expansion of a multiplicity-free subvariety (such as $Y_{\gamma}$) given the cohomology expansion.
This provides answers to Questions~1 and~2 in ordinary $K$-theory.

However, we are not aware of any formula (in ordinary cohomology or $K$-theory) that is geometrically natural from the perspective of 
Gr\"{o}bner degenerations of $I_{\gamma}$.

Question~\ref{question:D} in the case of
$\Upsilon_{\gamma}(X;Y)$ for matchless $\gamma$ is equivalent to
certain (yet unsolved) Schubert calculus problems. Once the matchless
case is solved, a formula for the general case can be obtained by applying the operators $\partial_i$.
These expansions involve coefficients in ${\mathbb Z}_{\geq 0}[y_{2}-y_1,\ldots,y_n-y_{n-1}]$.

\excise{\begin{Example}
Use this example to illustrate that the $K$-polynomials satisfy the isobarics.

\noindent\emph{Checking equalities with Macaulay 2:} (Eventually this will turn into an example.) 
Impose ${\tt deg}(z_{ij})=1$. Then the Hilbert series of the ideal,
as computed by Macaulay (and with numerator $(1-t)^{n^2}$) is a polynomial $p(t)$. The substitution $p(1-t)$ should give the
product of  expected Grothendieck polynomials (where $x_i=1$). For example, when $\gamma=+-+-$ then the Hilbert series is
$(t-1)^4(t+1)^2/(1-t)^{16}$. Then $p(t)=(t-1)^4 (t+1)^2$. Then $p(1-t)=t^4(2-t)^2$. This predicts that there are $4$ degree $4$
"+" diagrams, $4$ degree $5$ diagrams and $1$ degree $6$ diagrams. This is true.
\end{Example}}

\section{Singularities of the orbit closures}

We use a modification of $I_{\gamma}$ to compute 
measures of the singularities of $p\in Y_{\gamma}$.

\subsection{Representative points}

We pick representative points of each ${\mathcal O}_{\gamma}$ to work with. Call a permutation
$\sigma$ {\bf $\gamma$-shuffled} if it is an assignment of 
\begin{itemize}
\item $1,2,\ldots,p$ (in any order) to the vertices of $\gamma$ that
have a ``$+$'' or are the left end of an arc; and
\item $p+1,p+2,\ldots,n$ (in any order) to the remaining positions.
\end{itemize}
Now let $F_{\bullet}^{\gamma,\sigma}=\langle {\vec v}_1,\ldots, {\vec v}_n\rangle$ be the flag given by
\[ {\vec v}_i =
\begin{cases}
	e_{\sigma(i)} & \text{ if vertex $i$ is a sign or the right end of an arc} \\
	e_{\sigma(i)} + e_{\sigma(j)} & \text{ if $i$ and $j$ form an arc and $i<j$}.
\end{cases}
\]
We recall the following easy facts for convenience:
\begin{Lemma}
Let $\gamma\in {\tt Clan}_{p,q}$ be given.
\begin{itemize}
\item[(I)] $F_{\bullet}^{\gamma,\sigma} \in {\mathcal O}_{\gamma}$ for any $\gamma$-shuffled $\sigma$.
\item[(II)] $v(\gamma)$ is $\gamma$-shuffled.
\item[(III)] If $\gamma$ is matchless then the $T$-fixed points ${\mathcal O}_{\gamma}^T=\{F_{\bullet}^{\gamma,\sigma}|\mbox{$\sigma$ is $\gamma$-shuffled}\}$.
\item[(IV)] If $\gamma$ is not matchless then ${\mathcal O}_{\gamma}$ contains no $T$-fixed points.
\item[(V)] Every point of $Y_{\gamma}$ is locally isomorphic to some $F_{\bullet}^{\beta,\sigma}$ for some $\beta\prec\gamma$ and $\beta$-shuffled $\sigma$.
\item[(VI)] Let ${\mathcal P}$ be any upper-semicontinuous property of points of $Y_{\gamma}$. Then $Y_{\gamma}$
globally has property ${\mathcal P}$ if and only if some $T$-fixed point $F_{\bullet}^{\tau,\sigma}$ has property ${\mathcal P}$ for every matchless $\tau \prec \gamma$.
\end{itemize}
\end{Lemma}
\begin{proof}
(I) follows easily from a theorem of T.~Matsuki-T.~Oshima \cite{Matsuki-Oshima}
and A.~Yamamoto \cite{Yamamoto} that $\mathcal{O}_{\gamma}$ is precisely the set of flags $F_{\bullet}$ such that
\begin{itemize}
\item ${\rm dim}(F_i\cap E_p)= \gamma(i;+)$;
\item ${\rm dim}(F_i\cap E^q)= \gamma(i;-)$;
\item ${\rm dim}(\pi(F_i)+F_j)= j+\gamma(i;j)$.
\end{itemize}
(II) is immediate from the definitions.
For (III) clearly the ``$\supset$'' inclusion is obvious.  On the other hand, the set of $\gamma$-shuffled permutations is clearly a left coset in 
$S_p\times S_q \backslash S_n$, and so has order $|S_p\times S_q| = p!q!$.  This is precisely the number of $T$-fixed points contained in any closed $K$-orbit, as each is isomorphic to the flag variety for the group $K$.  Thus the inclusion is an equality.  For (IV), simply note that there are $\binom{n}{p}$ closed orbits, each containing $p!q!$ $T$-fixed points (as just noted), for a total of $\binom{n}{p} \cdot p!q! = n!$ $T$-fixed points contained in the closed orbits.  This means that no non-closed orbit can contain a $T$-fixed point.  (Alternatively, (IV) follows directly from \cite[Corollary 6.6]{Springer}.)  For (V),
the elements of $GL_p\times GL_q$ provide the isomorphisms. Finally, for (VI), the matchless clans are the minimal elements
of the closure order.
\end{proof}

Part (IV) contrasts with Schubert varieties where every point is locally isomorphic to a $T$-fixed point. 
However, in view of (VI) these points of $Y_{\gamma}$ are still of special interest.

\subsection{The patch ideal}
Given a permutation $\sigma$, let $M_{n,\sigma}$ be the specialization of the generic matrix $M_n$ obtained by setting $z_{ij}=1$
if $i=\sigma(j)$ and $z_{ij}=0$ if $j>\sigma^{-1}(i)$. For example, if $\sigma=1324$ then (now writing the permutation matrix for $\pi$
with a $1$ in position $(\pi(i),i)$):
\[
M_{4,1324}=
\begin{pmatrix}
		1 & 0 & 0 & 0 \\
		z_{2,1} & z_{2,2} & 1 & 0 \\
		z_{31} & 1 & 0 & 0 \\
		z_{4,1} & z_{4,2} & z_{4,3} & 1
	\end{pmatrix}.
\]
For a clan $\beta$, let $v=v(\beta)$ and
let $L_{\beta}$ be the lower triangular unipotent matrix defined by having
$1$'s in positions $(v(j),v(i))$ whenever $i<j$ is matched in $\beta$. Now define $M_{n,\beta}=L_{\beta}M_{n,v(\beta)}$. So
if for example $\beta=
\begin{picture}(30,10)
\put(0,0){\epsfig{file=matchingsSection4.eps, height=.4cm}}
\put(11,1){$+-$}
\end{picture}$
then $v(\beta)=1324$,
\[
L_{\beta}=
\begin{pmatrix}
		1 & 0 & 0 & 0 \\
		0 & 1 & 0 & 0 \\
		1 & 0 & 1 & 0 \\
		0 & 0 & 0 & 1
	\end{pmatrix} 
 \mbox{ \ \ \ and  \ \ \ $M_{n,\beta}=
\begin{pmatrix}
		1 & 0 & 0 & 0 \\
		z_{2,1} & z_{2,2} & 1 & 0 \\
		z_{31}+1 & 1 & 0 & 0 \\
		z_{4,1} & z_{4,2} & z_{4,3} & 1
	\end{pmatrix}$.}
\]
Finally, define the {\bf patch ideal} $I_{\gamma,\beta}$ of $Y_{\gamma}$ at $\beta$ to be generated
by the same polynomials as the $K$-orbit determinantal ideal except that $M_n$ is replaced by $M_{n,\beta}$
in the definition.
\begin{Example}
Suppose $\gamma=\begin{picture}(27,10)
\put(0,0){\epsfig{file=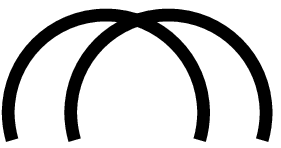, height=.4cm}}
\end{picture}$
and we continue with 
$\beta=
\begin{picture}(30,10)
\put(0,0){\epsfig{file=matchingsSection4.eps, height=.4cm}}
\put(11,1){$+-$}
\end{picture}$. Then the reader can check that $I_{\gamma,\beta}$ is generated by the determinant of 
\[
\begin{pmatrix}
		1 & 1 & 0 & 0 \\
		z_{2,1} & z_{2,1} & z_{2,2} & 1 \\
		0 & z_{31} & 1 & 0 \\
		0 & z_{4,1} & z_{4,2} & z_{4,3}
	\end{pmatrix}.
\]
\qed
\end{Example}

The following is standard; see the discussion of \cite{Insko.Yong}:
\begin{Proposition}
\label{prop:patch}
${\rm Spec}({\rm Fun}(M_{n,\beta})/I_{\gamma,\beta})$ is set-theoretically equal to a local neighbourhood of 
$Y_{\gamma}$ around the point $F_{\bullet}^{\beta,v(\beta)}$. (The point ${\bf 0}$ corresponds to $F_{\bullet}^{\beta,v(\beta)}$.)
\end{Proposition}
\begin{proof}
Let $g=L_{\beta}v(\beta)$. Then $gB_{-}B/B\cap Y_{\gamma}$ is an affine open neighbourhood of $Y_{\gamma}$ around $gB$.
Coordinates for $gB_{-}B/B$ are given by $M_{n,\beta}$. In view of Theorem~\ref{thm:wyserbruhat}, any matrix of $M_{n,\beta}$
representing a flag in $Y_{\gamma}$ must vanish on generators of $I_{\gamma,\beta}$ and conversely, by
Lemma~\ref{lemma:vanish}.
\end{proof}

\begin{Conjecture}
\label{conj:D}
$I_{\gamma,\beta}$ is a radical ideal.
\end{Conjecture}

Conjecture \ref{conj:D} has been verified for all patch ideals $I_{\gamma,\beta}$ with $\gamma \geq \beta$ through $p+q=6$.
Additionally, it has been verified exhaustively for patch ideals $I_{\gamma,\tau}$ with $\tau$ a \textit{matchless} clan for $n=7$,
as well as for the cases $(p,q) = (2,6)$ and $(3,5)$.  Numerous other successful checks of $I_{\gamma,\tau}$ with
$\tau$ matchless in the case $(p,q) = (4,4)$ have also been performed.

\subsection{$H$-polynomials and Kazhdan-Lusztig-Vogan polynomials}\label{sec:h-poly}
We propose an analogy between two families of polynomials, one of which 
are the Kazhdan-Lusztig-Vogan (KLV) polynomials.  

Standard references on KLV polynomials are \cite{Vogan,Lusztig-Vogan}.  In their most general form, these
polynomials are indexed by pairs $(Q,\mathcal{L})$ and $(Q',\mathcal{L}')$, where $Q,Q'$ are $K$-orbits on $G/B$,
and $\mathcal{L},\mathcal{L}'$ are $K$-equivariant local systems on $Q,Q'$, respectively.  For the associated
polynomials to be nonzero, the pairs $(Q,\mathcal{L})$ and $(Q',\mathcal{L}')$ must be related in 
\textit{$\mathcal{G}$-Bruhat order}, defined in \cite{Vogan}.  Since all $K$-equivariant local systems on all
orbits are trivial in the example we are considering, for us the KLV polynomials will be indexed simply by pairs
of orbits (or rather, by the corresponding pairs of clans) $\beta,\gamma$ such that 
$\caO_{\beta} \subseteq \overline{\caO_{\gamma}}$.  Furthermore, the
coefficient of $q^i$ in the polynomial $P_{\beta,\gamma}(q)$ measures the dimension of the $2i$-th intersection
homology group of $\overline{\caO_{\gamma}}$ in a neighborhood of a point of $\caO_{\beta}$, as follows from
\cite[Theorem 1.12]{Lusztig-Vogan}. This mirrors the relationship between Schubert varieties and ordinary
Kazhdan-Lusztig polynomials.

Consider the ${\mathbb Z}$-graded Hilbert series of ${\rm gr}_{{\mathfrak m}_p}{\mathcal O}_{p,Z}$, 
the associated graded ring of the local ring ${\mathcal O}_{p,Z}$ of a variety $Z$. This is denoted by 
${\rm Hilb}({\rm gr}_{{\mathfrak m}_p}{\mathcal O}_{p,Z},q)$.
The {\bf $H$-polynomial} $H_{p,Z}(q)$ is defined by
\[{\rm Hilb}({\rm gr}_{{\mathfrak m}_p}{\mathcal O}_{p,Z},q)=\frac{H_{p,Z}(q)}{(1-q)^{\dim Z}},\]
and $H_{p,Z}(1)$ is the {\bf Hilbert-Samuel multiplicity} ${\rm mult}_{p,Z}$.

\begin{Conjecture}
\label{conj:Hpoly}
\begin{itemize}
\item[(i)] ${\rm gr}_{{\mathfrak m}_p}{\mathcal O}_{p,Y_{\gamma}}$ is Cohen-Macaulay.
\item[(ii)] $H_{p,Y_{\gamma}}(q)\in {\mathbb Z}_{\geq 0}[q]$.
\item[(iii)] $H_{p,Y_{\gamma}}(q)\in {\mathbb Z}_{\geq 0}[q]$ is upper-semicontinuous.
\end{itemize}
\end{Conjecture}
In fact (i) implies (ii), by standard facts from commutative algebra. However, (i) and (ii) seem to be logically independent
of (iii). 

Properties (ii) and (iii) are true for the KLV polynomial $P_{\beta,\gamma}(q)$.  Property (ii) 
follows from \cite[Theorem 1.12]{Lusztig-Vogan}, while property (iii) holds due to recent work of W.M. McGovern
\cite{McGovern-preprint}.  Thus the above conjecture is our rationale for drawing an analogy between
$H_{\beta,\gamma}(q)$ and $P_{\beta,\gamma}(q)$. (Here, $H_{\beta,\gamma}(q)$ is the $H$-polynomial $H_{p,Y_{\gamma}}(q)$
where $p$ is any point of $\caO_{\beta} \subseteq Y_{\gamma}$.)  An analogous
analogy and conjecture was proposed in the Schubert variety setting 
in \cite{Li.Yong}. 

\begin{Theorem} 
\label{prop:PHineq}
If $\gamma$ is non-crossing then $H_{\beta,\gamma}(q)
\in {\mathbb Z}_{\geq 0}[q]$ and 
$P_{\beta,\gamma}(q)\leq 
H_{\beta,\gamma}(q)$ (coefficient-wise inequality).
\end{Theorem}
\begin{proof}
When $\gamma$ is non-crossing $Y_{\gamma}=X_{v(\gamma)}^{u(\gamma)}$. The KLV polynomial is the $IH$-Poincar\'{e} polynomial
at a point of $X_{v(\gamma)}^{u(\gamma)}$.
By \cite{KWY}, this is therefore the product of Kazhdan-Lusztig
polynomials for $X_{v(\gamma)}$ and for $X^{u(\gamma)}$.
The same is true for the $H$-polynomial.
However, $v(\gamma)$ and $u(\gamma)$ are vexillary. It is a theorem of 
\cite{Li.Yong} that for the Schubert varieties
involved, the $H$-polynomials have nonnegative coefficients and bound the
Kazhdan-Lusztig polynomials. Nonnegativity and this bound on polynomials 
is preserved by multiplication.
\end{proof}

\begin{Example}
The inequality of Theorem~\ref{prop:PHineq} does not always hold when $\gamma$ is not non-crossing. 
For example, if $\gamma=\begin{picture}(27,10)
\put(0,0){\epsfig{file=matchingsSection4KLV.eps, height=.4cm}}
\put(8,1){\tiny $+$}
\end{picture}$
then $P_{-+++-,\gamma}(q)=1+q^2$, as one can verify using ATLAS (\url{http://www.liegroups.org}).  However, we have $H_{-+++-,\gamma}(q)=1+q$.\qed
\end{Example}

A.~Woo and the first author have found an explicit combinatorial rule for $P_{\beta,\gamma}(q)$
when $\gamma$ is non-crossing.

The following also seems true:

\begin{Conjecture} 
\label{conj:reduced}
${\rm Spec}({\rm gr}_{{\mathfrak m}_p}{\mathcal O}_{p,Y_{\gamma}})$ is reduced.
\end{Conjecture}

Using the patch equations one can exhaustively check Conjectures~\ref{conj:Hpoly} and~\ref{conj:reduced}  for all $(p,q)$ where $p+q\leq 7$. We have also done checks for some larger cases.

\section*{Appendix}
Below we give the polynomials $\Upsilon_{\gamma}(X;Y)$ for all $\gamma\in {\tt Clans}_{2,2}$.

\begin{center}
\begin{tabular}{|c|l|}
\hline
$\gamma$ & $\Upsilon_{\gamma}(X;Y)$\\
\hline
$--++$ & $(x_2 - y_2) (x_2 - y_1) (x_1 - y_2) (x_1 - y_1)$\\
$-+-+$ & $(x_1 - y_2) (x_1 - y_1) (x_1 - y_4 - y_3 + x_2) (x_2 - y_1 + x_3 - y_2)$\\
$-++-$ &  $(x_1 - y_2) (x_1 - y_1) (-x_1 y_3 + y_4 y_3 + y_3^2  - x_2 y_3 + x_1 x_3 - y_4 x_3$\\
\ &  $- x_3 y_3 + x_2 x_3 + x_2 x_1 - y_4 x_2 - y_4 x_1 + y_4^2)$\\
$+--+$ & $(x_1 - y_4) (x_1 - y_3) (-x_1 y_2 + y_1 y_2 + y_2^2  - x_2 y_2 + x_1 x_3 - x_3 y_1$\\
\ &    $- x_3 y_2 + x_2 x_3 + x_2 x_1 - x_2 y_1 - y_1 x_1 + y_1^2)$\\
$+-+-$ & $(x_1 - y_4) (x_1 - y_3) (x_1 - y_1 - y_2 + x_2) (x_3 - y_3 - y_4 + x_2)$\\
$++--$ & $(x_2 - y_4) (x_2 - y_3) (x_1 - y_4) (x_1 - y_3)$\\
$\begin{picture}(30,10)
\put(10,0){\epsfig{file=matchingsSection4.eps, height=.4cm}}
\put(0,1){$-$}
\put(21,1){$+$}
\end{picture}$ & $(x_1 - y_2) (x_1 - y_1) (x_2 - y_1 + x_3 - y_2)$\\
$\begin{picture}(30,10)
\put(18,0){\epsfig{file=matchingsSection4.eps, height=.4cm}}
\put(0,1){$-+$}
\end{picture}$ & $(x_1 - y_2) (x_1 - y_1) (x_1 - y_4 - y_3 + x_2)$\\
$\begin{picture}(30,12)
\put(0,0){\epsfig{file=matchingsSection4.eps, height=.4cm}}
\put(11,1){$-+$}
\end{picture}$ & $(x_1 - y_4 - y_3 + x_2) (-x_1 y_2 + y_1 y_2 + y_2^2  - x_2 y_2 + x_1 x_3 - x_3 y_1$\\
\ & $- x_3 y_2 + x_2 x_3 + x_2 x_1 - x_2 y_1 - y_1 x_1 + y_1^2)$\\
$\begin{picture}(30,10)
\put(0,0){\epsfig{file=matchingsSection4.eps, height=.4cm}}
\put(11,1){$+-$}
\end{picture}$ & $(x_1 - y_1 - y_2 + x_2) (-x_1 y_3 + y_4 y_3 + y_3^2  - x_2 y_3 + x_1 x_3 - y_4 x_3$\\
\ & $-x_3 y_3 + x_2 x_3 + x_2 x_1 - y_4 x_2 - y_4 x_1 + y_4^2)$\\
$\begin{picture}(30,12)
\put(18,0){\epsfig{file=matchingsSection4.eps, height=.4cm}}
\put(0,1){$+-$}
\end{picture}$ & $(x_1 - y_4) (x_1 - y_3) (x_1 - y_1 - y_2 + x_2)$\\
$\begin{picture}(30,10)
\put(10,0){\epsfig{file=matchingsSection4.eps, height=.4cm}}
\put(0,1){$+$}
\put(21,1){$-$}
\end{picture}$ & $(x_1 - y_4) (x_1 - y_3) (x_3 - y_3 - y_4 + x_2)$\\
$\begin{picture}(30,10)
\put(10,0){\epsfig{file=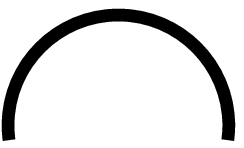, height=.4cm}}
\put(0,1){$-$}
\put(15,1){$+$}
\end{picture}$ & $(x_1 - y_2) (x_1 - y_1)$\\
$\begin{picture}(30,10)
\put(1,0){\epsfig{file=matchingstableA.eps, height=.4cm}}
\put(6,0){$-$}
\put(21,1){$+$}
\end{picture}$ & $-x_1 y_2 + y_1 y_2 + y_2^2  - x_2 y_2 + x_1 x_3 - x_3 y_1
- x_3 y_2 + x_2 x_3 + x_2 x_1 - x_2 y_1 - y_1 x_1 + y_1^2$\\
$\begin{picture}(30,10)
\put(1,0){\epsfig{file=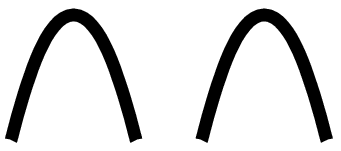, height=.4cm}}
\end{picture}$ & $(x_1 - y_4 - y_3 + x_2) (x_1 - y_1 - y_2 + x_2)$\\
$\begin{picture}(30,10)
\put(1,0){\epsfig{file=matchingstableA.eps, height=.4cm}}
\put(6,0){$+$}
\put(21,1){$-$}
\end{picture}$ & $-x_1 y_3 + y_4 y_3 + y_3^2  - x_2 y_3 + x_1 x_3 - y_4 x_3 - x_3 y_3 + x_2 x_3 + x_2 x_1
     - y_4 x_2 - y_4 x_1 + y_4^2$\\
$\begin{picture}(30,10)
\put(10,0){\epsfig{file=matchingstableA.eps, height=.4cm}}
\put(0,1){$+$}
\put(15,1){$-$}
\end{picture}$ & $(x_1 - y_4) (x_1 - y_3)$\\
$\begin{picture}(30,10)
\put(0,0){\epsfig{file=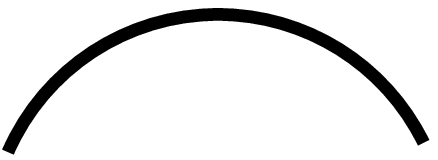, height=.4cm}}
\put(7,1){$-+$}
\end{picture}$ & $x_1 - y_1 - y_2 + x_2$\\
$\begin{picture}(30,10)
\put(5,0){\epsfig{file=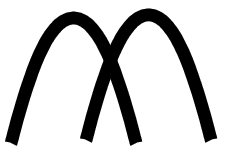, height=.4cm}}
\end{picture}$ & $x_2 - y_1 - y_3 + 2 x_1 - y_4 - y_2 + x_3$\\
$\begin{picture}(30,10)
\put(0,0){\epsfig{file=matchingstableB.eps, height=.4cm}}
\put(7,1){$+-$}
\end{picture}$ & $x_1 - y_4 - y_3 + x_2$\\
$\begin{picture}(30,10)
\put(0,0){\epsfig{file=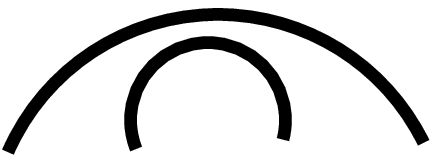, height=.4cm}}
\end{picture}$ & $1$\\
\hline
\end{tabular}
\end{center}

\section*{Acknowledgements}
We wish to thank Bill Graham, Allen Knutson, William McGovern, Oliver Pechenik, 
Hal Schenck, Peter Trapa, Hugh Thomas and Alexander Woo for helpful correspondence. We also thank the anonymous referee for his/her useful suggestions.
AY was supported by NSF grants. This text was completed while AY 
was a Helen Corley Petit scholar at UIUC.

\end{document}